\documentclass[12pt,final]{amsart}
 
\usepackage{a4wide}
\usepackage[notcite,notref,color]{showkeys} 

\usepackage{amsmath,amssymb,amsthm,bbm} 
\usepackage{enumitem} 
\usepackage[overload]{textcase}
 
\usepackage{epsfig,psfrag,color} 
\usepackage{verbatim}

\usepackage[colorlinks,bookmarks,linkcolor=black,citecolor=black]{hyperref} 

\setlist{itemsep=3pt,parsep=0pt,topsep=2pt,partopsep=0pt}  
\setlist{leftmargin=2.5\parindent} 

\def\itm#1{\rm ({#1})} 
\def\itmit#1{\itm{\it #1\,}} 
\def\rom{\itmit{\roman{*}}} 
\def\abc{\itmit{\alph{*}}}
\def\abcstar{\itmit{\alph{*}*}}
 
\def\endofClaim{\hfill\scalebox{.6}{$\Box$}}

\let\subset\subseteq  
\let\eps\varepsilon 
\let\rho\varrho 
\def\dcup{\cup}

\def\cB{\mathcal{B}}

\def\cD{\mathcal{D}}
\def\cF{\mathcal{F}}
\def\cG{\mathcal{G}}
\def\cH{\mathcal{H}}

\def\cM{\mathcal{M}}
\def\cP{\mathcal{P}}
\def\cS{\mathcal{S}}
\def\cT{\mathcal{T}}

\def\bfu{\mathbf{u}}

\def\le{\leqslant}
\def\ge{\geqslant}

\newtheorem{theorem}{Theorem}%[section] 
\newtheorem*{SzRL}{Szemer\'edi's Regularity Lemma}
\newtheorem*{SzRLdeg}{Szemer\'edi's Regularity Lemma (minimum degree form)}
\newtheorem*{Count}{Counting Lemma}
\newtheorem{lemma}[theorem] {Lemma}    
    
\newtheorem{prop}[theorem] {Proposition}   
\newtheorem{fact}[theorem]{Fact} 
\newtheorem{obs}[theorem] {Observation}  
\newtheorem{conj}[theorem]{Conjecture}
\newtheorem{prob}[theorem]{Problem} 
  
\theoremstyle{definition}
\newtheorem{definition}[theorem]{Definition}
\newtheorem{remark}[theorem]{Remark}
\theoremstyle{remark} 
\newtheorem{claim}[theorem]{Claim}  

\newcommand{\oldqed}{}

\newenvironment{claimproof}[1][Proof]{
  \renewcommand{\oldqed}{\qedsymbol}
  \renewcommand{\qedsymbol}{\endofClaim}
  \begin{proof}[#1]
}{
  \end{proof}
  \renewcommand{\qedsymbol}{\oldqed}
}
\newcommand{\By}[2]{\overset{\mbox{\tiny{#1}}}{#2}} 
\newcommand{\ByRef}[2]{   \By{\eqref{#1}}{#2} }

\newcommand{\eqByRef}[1]{ \ByRef{#1}{=} } 
\newcommand{\lByRef}[1]{  \ByRef{#1}{<} }

\newcommand{\geByRef}[1]{ \ByRef{#1}{\ge} }

\newcommand{\NATS}{\mathbb{N}} 

\newcommand{\Prob}{\mathbb{P}}

\def\Pr{\mathbb{P}}

\newcommand{\ds}{\displaystyle}

\newcommand{\ol}{\overline}
\newcommand{\girth}{\textup{girth}}
\newcommand{\Kn}{\textup{Kn}}
\newcommand{\Bor}{\textup{Bor}}
\newcommand{\BH}{\textup{BH}}

\newcommand{\Bin}{\textup{Bin}}
\newcommand{\Aut}{\textup{Aut}}

\newcommand{\tpl}[2]{\mathbf{#1}^{#2}}

\newcommand{\ang}{\measuredangle}

\DeclareMathOperator{\ex}{ex}

\title{The chromatic thresholds of graphs}

  \author[P. Allen]{Peter Allen}
  \author[J. B\"ottcher]{Julia B\"ottcher}
  \author[S. Griffiths]{Simon Griffiths}
  \author[Y. Kohayakawa]{Yoshiharu Kohayakawa}
  \author[R. Morris]{Robert Morris}
 \address{
    Peter Allen, Julia B\"ottcher, Yoshiharu Kohayakawa \hfill\break
    Instituto de Matem\'atica e Estat\'{\i}stica, Universidade de
    S\~ao Paulo, Rua do Mat\~ao 1010, 05508--090~S\~ao Paulo, Brasil.
  }
  \email{allen|julia|yoshi@ime.usp.br}
  \address{
    Robert Morris, Simon Griffiths \hfill\break
   IMPA, Estrada Dona Castorina 110, Jardim Bot\^anico, Rio de Janeiro, RJ, Brasil
 }
 \email{rob|sgriff@impa.br}

\thanks{
    PA was partially supported by FAPESP (Proc.~2010/09555-7);
    JB by FAPESP (Proc.~2009/17831-7);
    SG by CNPq (Proc.~500016/2010-2);
    YK by CNPq (Proc.~308509/2007-2);
    RM by a CNPq bolsa de Produtividade em Pesquisa.
    This research was supported by CNPq (Proc.~484154/2010-9).
    The authors are
    grateful to NUMEC/USP, N\'ucleo de Modelagem Estoc\'astica e Complexidade
    of the University of S\~{a}o Paulo, and Project MaCLinC/USP, for supporting
    this research. }

\date{\today}

\begin{document}

\begin{abstract}
  The chromatic threshold $\delta_\chi(H)$ of a graph~$H$ is the infimum
  of $d > 0$ such that there exists $C = C(H,d)$ for which
  every $H$-free graph~$G$ with minimum degree at least $d |G|$
  satisfies $\chi(G)\le C$. We prove that
  \[\delta_\chi(H) \, \in \, \left\{ \frac{r-3}{r-2}, \frac{2r-5}{2r-3}, \frac{r-2}{r-1} \right\}\]
  for every graph~$H$ with~$\chi(H)=r\ge 3$. We moreover characterise the
  graphs~$H$ with a given chromatic threshold, and thus determine
  $\delta_\chi(H)$ for every graph $H$. This answers a question of Erd\H{o}s and
  Simonovits [Discrete Math. 5 (1973), 323--334], and confirms a conjecture of
  {\L}uczak and Thomass\'e [preprint (2010), 18pp].
\end{abstract}

\maketitle

\section{Introduction}\label{sec:results}

Two central problems in Graph Theory involve understanding the structure of
graphs which avoid certain subgraphs, and bounding the chromatic number of
graphs in a given family. For more than sixty years, since Zykov~\cite{Zykov}
and Tutte~\cite{Tutte} first constructed triangle-free graphs with arbitrarily
large chromatic number, the interplay between these two problems has been an important area of study.
The generalisation of Zykov's result, by Erd\H{o}s~\cite{Erd59}, to $H$-free
graphs (for any non-acyclic $H$), was one of the first applications of the
probabilistic method in combinatorics.

In 1973, Erd\H{o}s and Simonovits~\cite{ES} asked whether such constructions are
still possible if one insists that the graph should have large minimum degree.
As a way of investigating their problem, they implicitly
 defined what is now known as the \emph{chromatic threshold} of a graph $H$ as
 follows (see~\cite[Section~4]{ES}):
\begin{multline*}
 \delta_\chi(H) \; := \; \inf \Big\{ d \,:\, \exists\, C = C(H,d) \text{
 such that if } G \textup{ is a graph on $n$ vertices, } \\ \text{ with }
 \delta(G) \ge d n \text{ and } H \not\subset G, \textup{ then } \chi(G) \le C\Big\}.
\end{multline*}
That is, for $d < \delta_\chi(H)$ the chromatic number of $H$-free graphs
with minimum degree $d n$ may be arbitrarily large, while for $d >
\delta_\chi(H)$ it is necessarily bounded. In this paper we shall determine
$\delta_\chi(H)$ for every graph $H$, and thus completely solve the problem of Erd\H{o}s and Simonovits.

The chromatic threshold has been most extensively investigated for the triangle
$H = K_3$, where in fact much more is known. Erd\H{o}s and Simonovits~\cite{ES}
conjectured that $\delta_\chi(K_3) = \frac{1}{3}$, which was proven in 2002 by
Thomassen~\cite{Thomassen02}, and that moreover if $G$ is triangle-free and
$\delta(G) > n/3$ then $\chi(G) \le 3$. This stronger conjecture was disproved
by H\"aggkvist~\cite{Hagg}, who found a $(10n/29)$-regular graph with chromatic
number four. However, Brandt and Thomass\'e~\cite{BraTho} recently showed that
the conjecture holds with $\chi(G)\le 3$ replaced by $\chi(G)\le 4$. Hence the
situation is now well-understood for triangle-free graphs $G$
(see~\cite{AES,Bra,BraTho,CJK,Hagg,Jin,LT}), and can be summarised as follows:
\begin{center}
\begin{tabular}{c|c|c|c|c}
$\delta(G) >$ \, & $2n/5$ & $10n/29$ & $n/3$ & $(1/3 - \eps)n$  \\[+0.6ex] \hline 
&&&& \\[-2.1ex]
$\chi(G) \le$ \, & \; 2 \; & \; 3 \; & \; 4 \; & \; $\infty$ 
\end{tabular}\\\
\end{center}
For bipartite graphs $H$, it follows trivially from the
K\"{o}v\'ari-S\'os-Tur\'an Theorem~\cite{KST} that $\delta_\chi(H) = 0$, and for
larger cliques, Goddard and Lyle~\cite{GodLyl} determined the chromatic
threshold, proving that $\delta_\chi(K_r) = \frac{2r-5}{2r-3}$ for every $r \ge
3$. Erd\H{o}s and Simonovits also conjectured that $\delta_\chi(C_5) = 0$, which
was proven (and generalised to all odd cycles) by Thomassen~\cite{Thomassen07}.

For graphs other than cliques and odd cycles, very little was known until
the recent work of Lyle~\cite{Lyle10} and the breakthrough of {\L}uczak and
Thomass\'e~\cite{LT} who introduced a new technique 
% (based on the VC-dimension of a hypergraph)
which allows the study of $\delta_\chi(H)$ for more general classes of graphs. 
In order to motivate their results, let us
summarise what was known previously for $3$-chromatic graphs~$H$. We have
seen that there are such graphs with chromatic threshold $0$ (the odd
cycle~$C_5$), and chromatic threshold $\frac{1}{3}$ (the triangle). A
folklore result is that there are also $3$-chromatic graphs with
chromatic threshold $\frac{1}{2}$, such as the \emph{octahedron}
$K_{2,2,2}$. Indeed, given a graph~$H$ with $\chi(H)=r\ge 3$, the
\emph{decomposition family} $\cM(H)$ of $H$ is the set of bipartite graphs
which are obtained from $H$ by deleting $r-2$ colour classes in some
$r$-colouring of $H$. Observe that $K_{2,2,2}$ has the property that its
decomposition family contains no forests (in fact,
$\cM(K_{2,2,2})=\{C_4\}$). It is not difficult to show
that whenever there are no forests in $\cM(H)$, the graph $H$
has chromatic threshold $\frac{1}{2}$ (see Proposition~\ref{noforest}).

Thus it remains to consider those $3$-chromatic graphs whose decomposition
family does contain a forest; in other words, graphs which admit a partition into a
forest and an independent set (such as all odd cycles). Lyle~\cite{Lyle10}
proved that these graphs have chromatic threshold strictly smaller than
$\frac{1}{2}$.  In fact, as we shall show, they have chromatic threshold at
most $\frac{1}{3}$ (see Theorem~\ref{mainthm}).
{\L}uczak and Thomass\'e~\cite{LT} described a large sub-family of these
graphs with chromatic threshold strictly smaller than~$\frac13$. More
precisely they considered triangle-free graphs which admit a partition
into a (not necessarily perfect) matching and an independent set; they called a graph (such as~$C_5$) \emph{near-bipartite} if
it is of this form, and proved that if $H$ is near-bipartite then~$\delta_\chi(H)=0$. 

However, {\L}uczak and Thomass\'e did not believe that
the near-bipartite graphs are the only graphs with chromatic threshold zero. Generalising
near-bipartite graphs, they defined~$H$ to be \emph{near-acyclic} if
$\chi(H)=3$ and~$H$ admits a partition into a forest~$F$ and an independent
set~$S$ such that every odd cycle of~$H$ meets~$S$ in at least two
vertices. Equivalently, for each tree~$T$ in~$F$ with colour classes $V_1(T)$ and
$V_2(T)$, there is no vertex of~$S$ with neighbours in both $V_1(T)$ and
$V_2(T)$ (see also Figure~\ref{fig:near-acyclic}). Observe that the
near-bipartite graphs are precisely the near-acyclic graphs in which every
tree is a single edge or vertex. The first graph in Figure~\ref{fig:examples} is
near-acyclic (as illustrated by the highlighted forest), but has no matching
in its decomposition family and thus is not near-bipartite.

\begin{figure}[t]
\psfrag{T1}{$T_1$}\psfrag{T2}{$T_2$}\psfrag{S}{$S$}
\includegraphics[width=5cm]{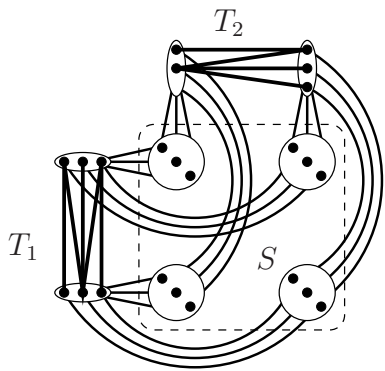} 
\caption{An illustration of
near-acyclic graphs with partition into an independent set~$S$ and a forest~$F$
consisting of two trees~$T_1$ and~$T_2$.}
\label{fig:near-acyclic}
\end{figure}

Lyle~\cite{Lyle10} proved that~$\delta_\chi(H)=0$ for a sub-family of near-acyclic
graphs~$H$ which are not necessarily near-bipartite, and
{\L}uczak and Thomass\'e gave a construction (see Section~\ref{borsuksec})
showing that every graph which is \emph{not} near-acyclic has chromatic
threshold at least $\frac{1}{3}$. They made the following conjecture.

\begin{conj}[\L uczak and Thomass\'e~\cite{LT}]\label{LTconj}
Let $H$ be a graph with $\chi(H) = 3$. Then $\delta_\chi(H) = 0$ if and only if $H$ is near-acyclic.
\end{conj}

We shall prove Conjecture~\ref{LTconj}, and moreover
determine~$\delta_\chi(H)$ for \emph{every} graph~$H$. In this theorem, we use the following generalisation of near-acyclic
graphs. We call a graph $H$ \emph{$r$-near-acyclic} if $\chi(H) = r\ge 3$, and there exist $r - 3$
independent sets in $H$ whose removal yields a near-acyclic graph.
Note in particular that if $H$ is $r$-near-acyclic, then there is a
forest in $\cM(H)$. 
Our main theorem is as follows.

\begin{theorem}\label{mainthm}
Let~$H$ be a graph with $\chi(H) = r \ge 3$. Then 
\[\delta_\chi(H) \, \in \,
\left\{ \frac{r-3}{r-2}, \, \frac{2r-5}{2r-3}, \, \frac{r-2}{r-1} \right\}\,.\]
Moreover, $\delta_\chi(H) \neq \frac{r-2}{r-1}$ if and only if $H$ has a forest
in its decomposition family, and $\delta_\chi(H) = \frac{r-3}{r-2}$ if and only
if $H$ is $r$-near-acyclic.
\end{theorem}

\begin{figure}[t]
\includegraphics[width=4cm]{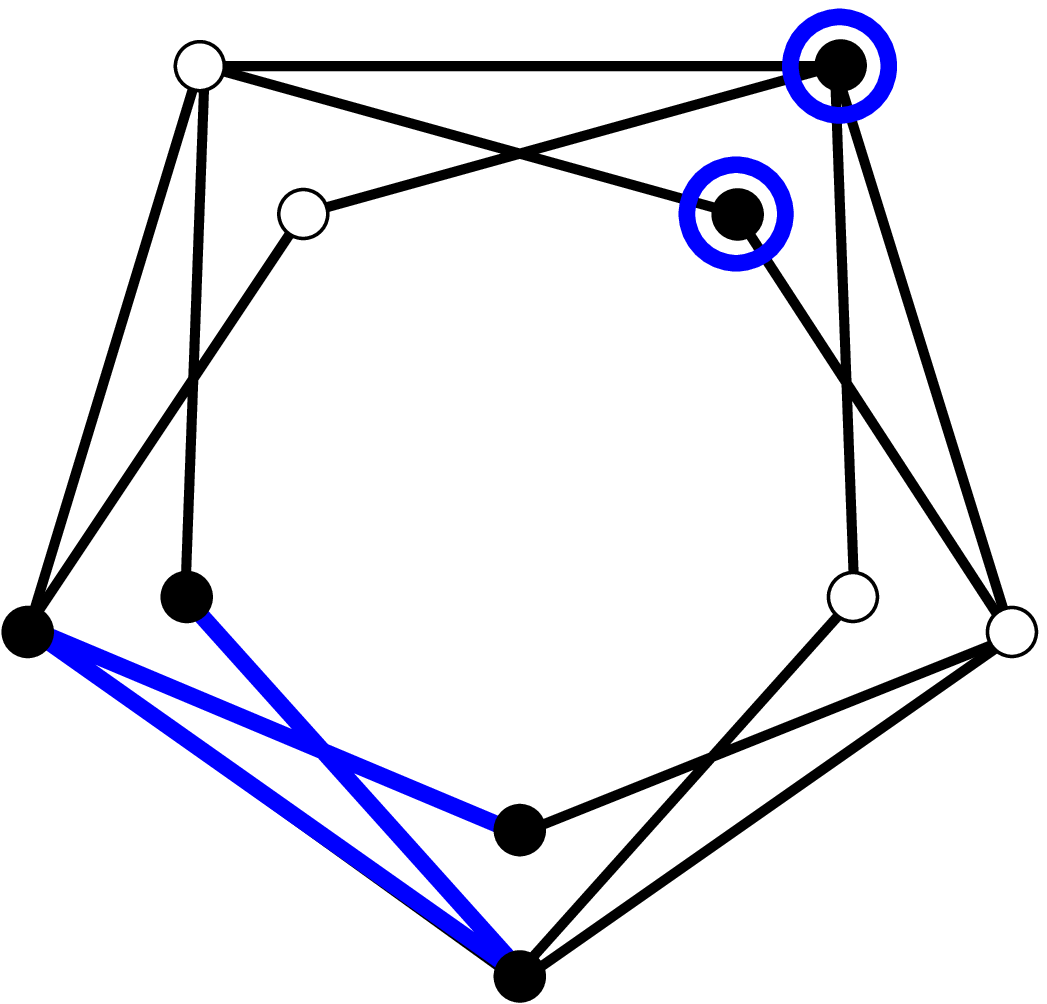}\qquad
\includegraphics[width=4cm]{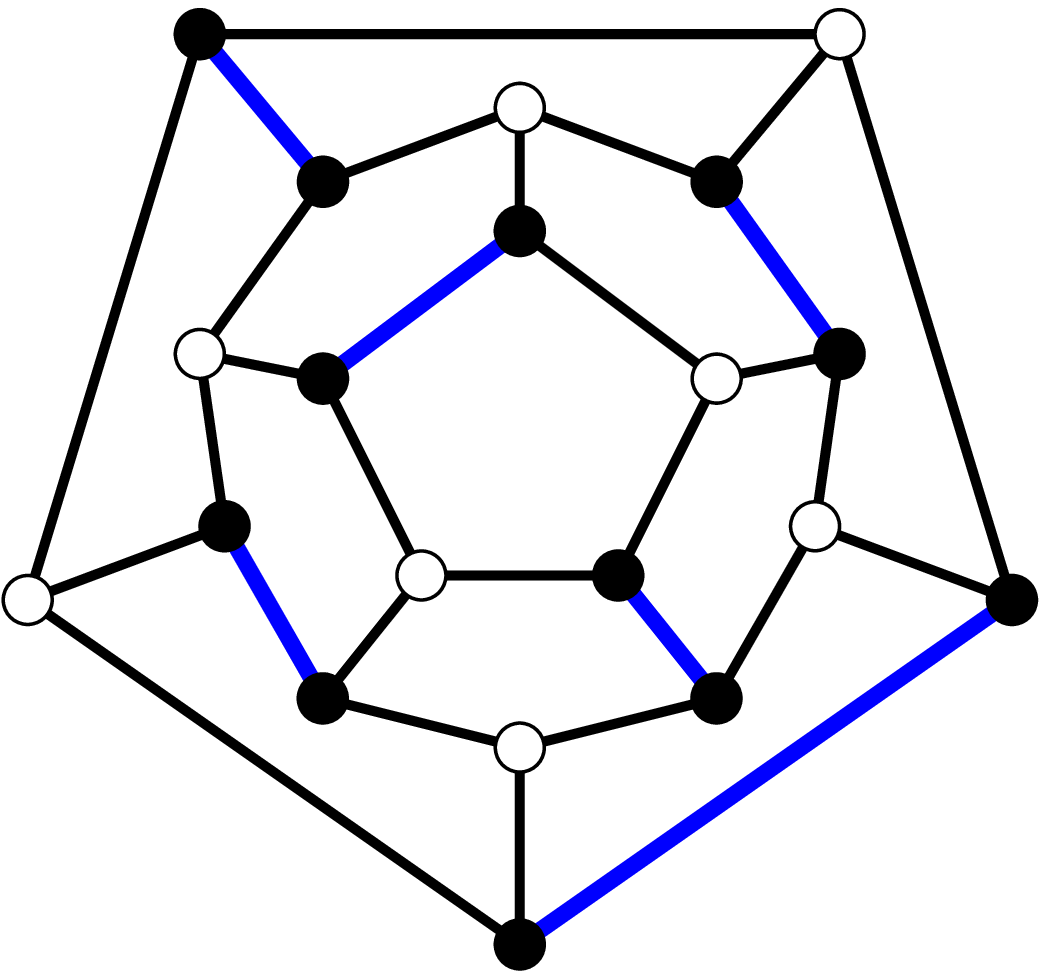}\qquad
\includegraphics[width=4cm]{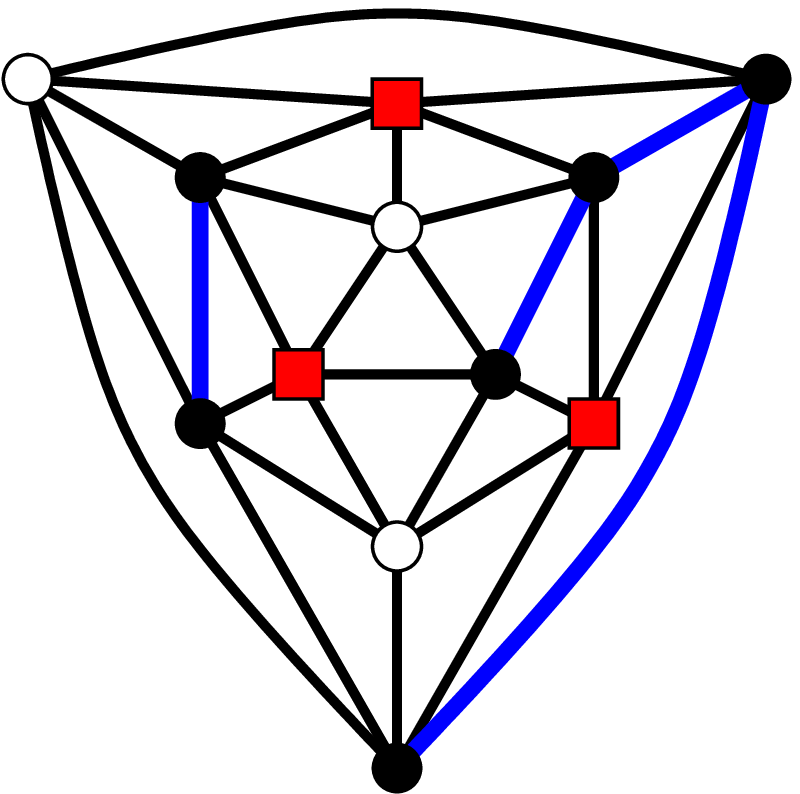}
\caption{A near-acyclic graph that is not near-bipartite, the dodecahedron, and
the icosahedron.}
\label{fig:examples}
\end{figure}

For example, the \emph{dodecahedron} is $3$-chromatic and near-bipartite,
hence it has chromatic threshold~$0$ (Figure~\ref{fig:examples} shows the
dodecahedron together with a corresponding matching). The \emph{icosahedron} on
the other hand has chromatic threshold $\frac35$ because it is four-chromatic
and has a forest in its decomposition family (a partition of the icosahedron
into this forest and two independent sets is also displayed in
Figure~\ref{fig:examples}), but is not $4$-near-acyclic.

For easier reference, given $H$ with $\chi(H)=r$, we define
\[\theta(H) := \frac{r-3}{r-2}\,,\quad\lambda(H) := \frac{2r - 5}{2r-3}\quad\text{and}\quad\pi(H) := \frac{r-2}{r-1}\,.\]
Observe that $\pi(H)$ is precisely the Tur\'an density of $H$, and therefore the Erd\H{o}s-Stone Theorem~\cite{ErdosStone1946} yields $\delta_\chi(H)\le\pi(H)$ for all $H$. Furthermore, the constructions giving the lower bounds in Theorem~\ref{mainthm}
are straightforward extensions of those given in~\cite{LT,Lyle10}. It follows that our main challenge is to prove that $\delta_\chi(H) \le \lambda(H)$ when $\cM(H)$ contains a forest, and that $\delta_\chi(H) \le \theta(H)$ when $H$ is $r$-near-acyclic. 

\smallskip

The recent results of both Lyle~\cite{Lyle10} and {\L}uczak and Thomass\'e~\cite{LT} contain important new techniques, which we re-use and extend here. Most significantly, {\L}uczak and Thomass\'e~\cite{LT} introduced a concept which they termed \emph{paired VC-dimension}, which is based on the classical Vapnik-\v{C}ervonenkis dimension of a set-system~\cite{VC}. Our proof of Conjecture~\ref{LTconj} relies on an extension of this technique (see Section~\ref{VCsec}), together with a new embedding lemma (see Section~\ref{Zykovsec}) which allows us to find a copy of $H$ in sufficiently many `well-structured' copies of the `Zykov graph', which is a universal near-bipartite graph. 

Lyle~\cite{Lyle10} introduced a novel graph partitioning method based on the celebrated Szemer\'edi Regularity Lemma. We shall use his partition in Section~\ref{Lylesec}, together with averaging arguments similar to those in~\cite{Allen10}, to prove that $\delta_\chi(H)\le\lambda(H)$ for any graph $H$ such that $\cM(H)$ contains a forest. In Section~\ref{proofsec} we shall combine and extend
both techniques in order to generalise Conjecture~\ref{LTconj} to arbitrary $r \ge 3$. 

\smallskip

\paragraph{\bf Organisation}
In Section~\ref{toolsec} we state the Regularity Lemma in the form in which we shall use it, together with some
auxiliary tools, and provide some notes on notation. In Section~\ref{Lylesec} we
prove that $\delta_\chi(H)\le\lambda(H)$ for any graph $H$ such that $\cM(H)$
contains a forest. In addition we give a construction which shows that if
$\cM(H)$ does not contain a forest, then $\delta_\chi(H)\ge\pi(H)$. In
Section~\ref{borsuksec} we provide a construction (using the Borsuk-Ulam
Theorem) which shows that for graphs~$H$ which are not $r$-near-acyclic, we have
$\delta_\chi(H)\ge\lambda(H)$. In Section~\ref{Zykovsec} we introduce a generalisation of
 the class of Zykov graphs, a class of universal near-bipartite graphs which were used in~\cite{LT}. We show that 
for every near-acyclic graph $H$, if $G$ contains a suitably well-structured
collection of Zykov graphs, then $G$ contains $H$. Complementing this,
in Section~\ref{VCsec} we refine {\L}uczak and Thomass\'e's paired VC-dimension
argument to show that every graph with linear minimum degree and sufficiently
large chromatic number indeed contains such a well-structured collection of
Zykov graphs. Completing the proof of Theorem~\ref{mainthm}, in
Section~\ref{proofsec} we give a construction which shows that for any
$r$-chromatic graph $H$ we have $\delta_\chi(H)\ge\theta(H)$, and combine the
results of Sections~\ref{Zykovsec} and~\ref{VCsec} with the Regularity Lemma
in order to show that for $r$-near-acyclic graphs $H$,
we have $\delta_\chi(H)\le\theta(H)$. Finally, in Section~\ref{probsec} we
conclude with a collection of open problems.

\section{Tools and the Regularity Lemma}\label{toolsec}

In this section we shall state some of the tools used in the proof of
Theorem~\ref{mainthm}. In particular, we shall recall the Szemer\'edi
Regularity Lemma, which is one of the most powerful and important results
in Graph Theory. Introduced in the 1970s by Szemer\'edi~\cite{SzRL} in
order to prove that sets of positive density in the integers contain
arbitrarily long arithmetic progressions (a result known as Szemer\'edi's
Theorem~\cite{Sz73}), it says (roughly) that \emph{any} graph can be approximated well
by a bounded number of `quasi-random' graphs. The lemma has turned out to
have an enormous number of applications, and many important extensions and
variations have been proved (see for
example~\cite{Gowers,KRSS,LovSze07,RodSch10} and the references therein).
The reader who is unfamiliar with the Regularity Lemma is encouraged to see the excellent surveys~\cite{KS93,KSSS}.

We begin by stating the Regularity Lemma in the form in which we shall use
it. Let $(A,B)$ be a pair of subsets of vertices
of a graph~$G$.
We write $d(A,B) = \frac{e(A,B)}{|A||B|}$,
and call $d(A,B)$ the \emph{density} of the pair $(A,B)$.
(Here $e(A,B)$ denotes the number of edges with one endpoint
in $A$ and the other in $B$.)
For each $\eps > 0$, we say that $(A,B)$ is \emph{$\eps$-regular} if
$| d(A,B) - d(X,Y) | < \eps$
for every $X \subset A$ and $Y \subset B$ with $|X| \ge \eps |A|$ and $|Y|
\ge \eps |B|$. 

A partition $V_0 \cup V_1 \cup \ldots \cup V_k$ of $V(G)$ is said to be an
\emph{$\eps$-regular partition} (or sometimes a Szemer\'edi partition of $G$ for
$\eps$) if $|V_0| \le \eps n$, $|V_1| = \ldots = |V_k|$, and all but at most
$\eps k^2$ of the pairs $(V_i,V_j)$ are $\eps$-regular. We will often refer
to the partition classes $V_1,\ldots,V_k$ as the \emph{clusters} of the
regular partition. In its simplest form, the Regularity Lemma is as follows.

\begin{SzRL}
  For every $\eps > 0$ and every $k_0 \in \NATS$, there exists a constant
  $k_1 = k_1(k_0,\eps)$ such that the following holds. Every graph $G$ on
  at least $k_1$ vertices has an $\eps$-regular partition into $k$ parts,
  for some $k_0 \le k \le k_1$.
\end{SzRL}

We shall in fact use a slight extension of the statement above, which follows easily from~\cite[Theorem~1.10]{KS93} (a proof can be found in, e.g.,~\cite[Proposition 9]{KOTPlanar}). Given $0 < d < 1$ and a pair $(A,B)$ of sets of vertices in a graph $G$, we say that $(A,B)$ is \emph{$(\eps,d)$-regular} if it is $\eps$-regular and has density at least $d$. 

Given an $\eps$-regular partition $V_0 \cup V_1\cup \dots \cup V_k$ of $V(G)$
and $0<d<1$, we define a graph~$R$, called the \emph{reduced graph} of the
partition, as follows: $V(R) =[k]= \{1,\ldots,k\}$ and $ij \in E(R)$ if and only
if $(V_i,V_j)$ is an $(\eps,d)$-regular pair. We shall occasionally omit
the partition, and simply say that $G$ has \emph{$(\eps,d)$-reduced
  graph}~$R$.

\begin{SzRLdeg}
Let $0 < \eps < d < \delta < 1$, and let $k_0 \in \NATS$. There exists a constant $k_1 = k_1(k_0,\eps,\delta,d)$ such that the following holds. Every graph $G$ on $n > k_1$ vertices, with minimum degree $\delta(G) \ge \delta n$, has an $(\eps,d)$-reduced graph~$R$ on~$k$ vertices, with $k_0 \le k \le k_1$ and $\delta(R) \ge \big( \delta - d - \eps \big)k$.
\end{SzRLdeg}

Thus the reduced graph~$R$ of~$G$ `inherits' the high minimum degree of~$G$. The main motivation for the definition of $\eps$-regularity is the following so-called `counting lemma' (see~\cite[Theorem~3.1]{KS93}, for example).

\begin{Count}
Let $G$ be a graph with $(\eps,d)$-reduced graph $R$ whose clusters contain
$m$ vertices, and suppose that there is a homomorphism $\phi\colon H\to R$.
Then $G$ contains at least
\[\frac{1}{\big|\Aut(H)\big|}\big(d-\eps|H|\big)^{e(H)}m^{|H|}\] copies of $H$,
each with the property that every vertex $x\in V(H)$ lies in the cluster corresponding to the vertex $\phi(x)$ of $R$.
\end{Count}

Note that since we count unlabelled copies of $H$, it is necessary to correct
for the possibility that two different maps from $H$ to $G$ may yield the same
copy of $H$ (precisely when they differ by some automorphism of $H$). In fact,
such an automorphism of $H$ must also preserve $\phi$, but dividing by the
number of elements of the full automorphism group $\Aut(H)$ provides a lower
bound which is sufficient for our purposes. We state one more useful fact about subpairs of $(\eps,d)$-regular pairs.
% This is known as the Slicing Lemma.

\begin{fact}\label{prop:subpair}
Let $(U,W)$ be an $(\eps,d)$-regular pair and $U'\subset U$, $W'\subset W$ satisfy $|U'|\ge \alpha|U|$ and $|W'|\ge\alpha|W|$. Then $(U',W')$ is $(\eps/\alpha,d-\eps)$-regular.
\end{fact}

We shall also use the following straightforward and well-known fact several times.

\begin{fact}\label{prop:forest}
Let~$F$ be a forest and~$G$ be a graph on $n\ge 1$ vertices. If $e(G) \ge |F|n$,
then $F\subset G$.
\end{fact}

\begin{proof}
Since $G$ has average degree at least $2|F|$, it contains a subgraph~$G'$ with minimum degree at least $|F|$. It is easy to show that $G'$ contains~$F$; for example, remove a leaf and apply induction.
\end{proof}

\subsection{Notation}

We finish this section by describing some of the notation which we shall
use throughout the paper. Most is standard (see~\cite{MGT}, for example);
we shall repeat non-standard definitions when they are first used.

For each $t \in \NATS$, let $[t] = \{1,\ldots,t\}$. 
% $Y^{(t)}$ denotes the collection of subsets of $Y$ of size $t$. 
We say that we \emph{blow up a vertex $v \in V(G)$ to size~$t$} if we replace
$v$ by an independent set of size~$t$, and replace each edge containing $v$ by a
complete bipartite graph. Given disjoint sets $X$ and $Y$, 
we shall write $K[X,Y]$ for the edge set of the complete bipartite graph on $X \cup Y$,
that is, the set of all pairs with one end in $X$ and the other in $Y$. We write
$K_s(t)$ for the complete $s$-partite graph with $t$ vertices in each part: that
is, the graph obtained from $K_s$ by blowing up each of its vertices to
size~$t$.

Given a graph $G$, we write $E(G)$ for the set of edges of $G$, and $e(G)$ for
$|E(G)|$. We use both $|G|$ and $v(G)$ to denote the number of vertices of
$G$. Given a set $X \subset V(G)$, we write $E(X)$ for the set of edges of~$G$ with both ends in~$X$, and $N(X)$ for the set of \emph{common} neighbours of the vertices in $X$.
If $D \subset E(G)$, then $N(D)$ denotes the set of common neighbours of
$V(D)$, the set of the endpoints of edges in $D$. (In particular, if $e = xy$
is an edge then $N(e) = N\big(\{x,y\}\big)=N(x) \cap N(y)$.)
Further, we let $G[D]$ denote the subgraph of~$G$ with vertex set $V(D)$ and edge
set~$D$, and write $\ol{d}(D)$ for the average degree of $G[D]$ and
$\delta(D)$ for the minimum degree of $G[D]$. 

A graph $G$ is said to be \emph{$C$-degenerate} if there exists an ordering
$(v_1,\ldots,v_n)$ of $V(G)$ such that $v_{k+1}$ has at most $C$ neighbours
in $\{v_1,\ldots,v_k\}$ for every $1\le k\le n-1$. Finally, if $e_1,\dots,e_\ell
\in E(G)$, then we shall write $\tpl{e}{\ell}$ for the tuple $(e_1,\dots,e_\ell)$.

\section{Graphs with large chromatic threshold}\label{Lylesec}

In this section we shall categorise the graphs $H$ with chromatic threshold
greater than $\lambda(H)$. First observe that a trivial upper bound on
$\delta_\chi(H)$ is given by the \emph{Tur\'an density} of~$H$,
\begin{equation*}
  \pi(H) \, = \, \lim_{n \to \infty} \frac{\ex(n,H)}{\binom{n}{2}} \, = \, \frac{\chi(H)-2}{\chi(H)-1},
\end{equation*}
since if $\delta(G)=\delta n$ with $\delta > \pi(H) n$ then $H \subset G$, by
the Erd\H{o}s-Stone Theorem~\cite{ErdosStone1946}. Moreover, it is not hard to prove the
following sufficient condition for equality, which can be found, for example,
in~\cite{Lyle10}. Recall that $\cM(H)$ denotes the decomposition family of
$H$.

\begin{prop}\label{noforest}
Let $H$ be a graph with $\chi(H) = r \ge 3$. If $\cM(H)$ does not contain a
forest, then $\delta_\chi(H) = \pi(H)=\frac{r-2}{r-1}$.
\end{prop}

For each $k,\ell \in \NATS$, we shall call a graph $G$ a
\emph{$(k,\ell)$-Erd\H{o}s graph} if it has chromatic number at least $k$, and
girth (length of the shortest cycle) at least $\ell$. In one of the first
applications of the probabilistic method, Erd\H{o}s~\cite{Erd59} proved that
such graphs exist for every $k$ and $\ell$.

\begin{proof}[Proof of Proposition~\ref{noforest}] Let $H$ be a graph with
$\chi(H)=r\ge 3$ such that $\cM(H)$ contains no forest, let $C \in \NATS$, and
let $G'$ be a $(C,|H|+1)$-Erd\H{o}s graph; that is, $\chi(G') \ge C$ and
$\girth(G') \ge |H| + 1$. Let~$G$ be the graph obtained from the complete,
balanced $(r-1)$-partite graph on $(r-1)|G'|$ vertices by replacing one of its
partition classes with~$G'$. Then $\delta(G) = \frac{r-2}{r-1}v(G)$,  $H
\not\subset G$, and $\chi(G) \ge C$.
\end{proof}

We remark that the same construction, with the complete balanced $(r-1)$-partite
graph replaced by a complete balanced $(r-2)$-partite graph, shows that,
whatever the structure of $H$, its chromatic threshold is at
least~$\frac{r-3}{r-2}$ (see Proposition~\ref{lower}).

Lyle~\cite{Lyle10} showed that the condition of Proposition~\ref{noforest} is necessary.

\begin{theorem}[Lyle~\cite{Lyle10}]\label{lylethm}
If $\chi(H) = r\ge 3$, then $\delta_\chi(H) < \pi(H)=\frac{r-2}{r-1}$ if and
only if the decomposition family of~$H$ contains a forest.
\end{theorem}

In this section we shall strengthen this result by proving that if
$\delta_\chi(H) < \pi(H)$, then it is at most $\lambda(H)$.

\begin{theorem}\label{thm:forest}
Let $H$ be a graph with $\chi(H) = r \ge 3$. If $\cM(H)$ contains a forest, then  
$$\delta_\chi(H) \, \le \, \frac{2r-5}{2r-3}.$$
\end{theorem}

The proof of Theorem~\ref{thm:forest} is roughly as follows. Let $\gamma > 0$,
and let $G$ be a sufficiently large graph with \[\delta(G) \, \ge \, \left(
\frac{2r-5}{2r-3} + \gamma \right)v(G).\]
For some suitably
small $\eps$ and $d$, let $V_0 \cup \ldots \cup V_k$ be the partition of $V(G)$
given by the minimum degree form of the Szemer\'edi Regularity Lemma, and $R$ be
the $(\eps,d)$-reduced graph of this partition. Define, for each $I \subset
[k]$, $$X_I \, := \, \Big\{ v\in V(G) \,\colon\, i \in I \Leftrightarrow |N(v)
\cap V_i| \ge d|V_i| \Big\}.$$ We remark that this partition
was used by Lyle~\cite{Lyle10}. We show that $\chi\big(G[X_I]\big)$ is bounded
if $H\not\subset G$. We distinguish two cases. If $|I| \ge (2r-4)/(2r-3)$, then
it is straightforward to show that $R[I]$ contains a copy of $K_{r-1}$, and hence, by the Counting Lemma, that $N(x)$ contains `many'
(i.e., a positive density of) copies of $K_{r-1}(t)$ for every $x \in X_I$. We
then use the pigeonhole principle (see Lemma~\ref{pigeon}, below), to show that
either $|X_I|$ is bounded, or $H \subset G$.

If $|I| \le (2r-4)/(2r-3)$, then set $V_I = \bigcup_{j \in I} V_j$, and observe
that every pair $x,y \in X_I$ has `many' common neighbours in $V_I$. We use a
greedy algorithm (in the form of Lemma~\ref{pigeon}\ref{pigeon:a}) to conclude
that every edge is contained in a positive density of copies of $K_r$. Finally, we
shall use a counting version of a lemma of Erd\H{o}s (Lemma~\ref{lem:erdos})
together with the pigeonhole principle to show that when $H \not\subset G$,
$\chi(G[X_I])$ is bounded (see Lemma~\ref{lem:extend}).

\smallskip

We begin with some preliminary lemmas. The following lemma
from~\cite{Allen10} will be an important tool in the proof; it is a counting
version of a result of Erd\H{o}s~\cite{Erdos64}.

\begin{lemma}[Lemma~7 of~\cite{Allen10}]\label{lem:erdos}
For every $\alpha > 0$ and $s,t \in \NATS$ there is an $\alpha' =
\alpha'(\alpha,s,t) > 0$ such that the following holds. Let $G$ be a graph
on~$n$ vertices with at least $\alpha n^s$ copies of~$K_s$. Then $G$ contains
at least $\alpha'n^{st}$ copies of $K_s(t)$.
\end{lemma}

We shall also use the following easy lemma, which is just an application of a
greedy algorithm and the pigeonhole principle. Let $G + H$ denote the graph
obtained by taking disjoint copies of $G$ and $H$, and adding a complete bipartite graph between the two.

\begin{lemma}\label{pigeon}
  Let $\alpha,\delta > 0$ and $s,t \in \NATS$, let~$F$ be a forest, and
  suppose that~$H \subset F + K_s(t)$. Let $G$ be a graph on~$n$ vertices,
  and $X \subset V(G)$.
  \begin{enumerate}[label=\abc]
  \item\label{pigeon:a} If $\delta(G) \ge \delta n$ and
    \[|X| \ge \big(\alpha^{1/s} s + (1 - \delta)(s-1) \big) n\,,\] then
    $G[X]$ contains at least $\alpha n^s$ copies of~$K_s$.
  \item\label{pigeon:b} If $G\big[N(x)\big]$ contains at least $\alpha
    n^{(s+1)|H|}$ copies of $K_{s+1}(|H|)$ for every $x \in X$, then either
    $H \subset G$ or $|X| \le |H|/\alpha$.
  \end{enumerate}
\end{lemma}

\begin{proof}
  For part~\ref{pigeon:a}, we construct copies of~$K_s$ in $G[X]$ using the
  following greedy algorithm: First choose an arbitrary vertex $x_1 \in
  X$, then a vertex $x_2 \in X$ in the neighbourhood of $x_1$, then~$x_3 \in
  N(x_1) \cap N(x_2) \cap X$, and so on, until we find~$x_s \in X$ in the
  common neighbourhood of $x_1,\dots,x_{s-1}$. Clearly,
  $G\big[\{x_1,\dots,x_s\}\big]$ is a copy of~$K_s$.

  Now we simply count: for choosing~$x_i$ we have at least
  \[|X| - (i-1)(1-\delta)n \ge \alpha^{1/s}s \cdot n \ge (s!\alpha)^{1/s}
  \cdot n\] possibilities, so in total we have at least $s!\alpha n^s$
  choices. Since the algorithm can construct a particular copy of~$K_s$ at
  most in~$s!$ different ways, we have found at least $\alpha n^s$ distinct
  $K_s$-copies in $G[X]$.

  For part~\ref{pigeon:b}, simply observe that, by the pigeonhole
  principle, there is a copy~$T$ of $K_{s+1}(|H|)$ in~$G$ such that $T
  \subset G\big[N(x)\big]$ for at least $\alpha|X|$ vertices of~$X$. Since
  $H \subset K_{s+2}(|H|)$ this implies that either $H \subset G$ or
  $\alpha|X| < |H|$.
\end{proof}

The following result follows easily from Lemma~\ref{lem:erdos}. For a forest
$F$ and $H\subset F+K_s(t)$, it will enable us to draw conclusions about the
chromatic number of an $H$-free graph which contains many $K_{s+2}$-copies
arranged in a suitable way.

\begin{lemma}\label{lem:extend}
For every $\alpha > 0$ and $s,t \in \NATS$, there exists $\alpha' =
\alpha'(\alpha,s,t)$ such that for every forest $F$ and every graph $H \subset F
+ K_s(t)$, the following holds. Let $G$ be an $H$-free graph on~$n$ vertices, and
let $X \subset V(G)$ be such that every edge $xy \in E\big(G[X]\big)$ is
contained in at least $\alpha n^s$ copies of $K_{s+2}$ in~$G$.

Then~$G[X]$ is $(2|F|/\alpha')$-degenerate, and hence $\chi\big(G[X]\big) \le
(2|F|/\alpha')+1$.
\end{lemma}

Given a subgraph $K$ of $G$, we say that an edge~$e = xy$ of~$G$
\emph{extends to $e + K$} if $V(K)\subset N(x,y)$. We say $e$ \emph{extends to
a copy of $e+K_s(t)$} if $e$ extends to
$e+K$ for some copy $K$ of $K_s(t)$ in $G$.

\begin{proof}[Proof of Lemma~\ref{lem:extend}] Let~$\alpha' =
\alpha'(\alpha,s,t)$ be the constant provided by Lemma~\ref{lem:erdos}, let $xy
\in E(G[X])$ and let $G' = G[N(x) \cap N(y)]$. Then, by our assumption, $G'$
contains at least $\alpha n^{s}$ copies of~$K_s$. By Lemma~\ref{lem:erdos}, it
follows that $G'$ contains $\alpha' n^{st}$ copies of~$K_s(t)$, so $xy$ extends
to at least $\alpha'  n^{st}$ copies of~$e + K_s(t)$ in~$G$.

Now, let $x_1,\dots,x_{|X|}$ be an ordering of the vertices of~$X$ with the
property that $x_i$ has minimum degree in $G_i := G[X \setminus
\{x_1,\dots,x_{i-1}\}]$ for each $1\le i\le |X|$. In order to show that $G[X]$
is $(2|F|/\alpha')$-degenerate, it suffices to prove that $$e(G_i) \, \le \,
\frac{|F|}{\alpha'} \cdot |G_i|,$$ since then $\delta(G_i) \le 2|F|/\alpha'$, as
desired.

Since each edge $e \in E(G_i)$ extends to  at least $\alpha' n^{st}$ copies of
$e + K_s(t)$, it follows, by the pigeonhole principle, that there is a copy~$K'$
of $K_s(t)$ and a set $E_i \subset E(G_i)$ with $|E_i| \ge \alpha' e(G_i)$, such
that $e$ extends to~$e + K'$ for every $e \in E_i$. Let~$G^*_i$ be the graph
with vertex set $V(G_i)$ and edge set~$E_i$. Since $H \not \subset G$, it
follows that $F \not\subset G^*_i$.  Thus, by Fact~\ref{prop:forest}, we have
$$|F| \cdot |G_i| \, > \, e(G^*_i) \, \ge \, \alpha' e(G_i),$$ as required.
\end{proof}

We are ready to prove Theorem~\ref{thm:forest}. We shall apply the minimum
degree form of the Szemer\'edi Regularity Lemma, together with the Counting
Lemma and Lemmas~\ref{pigeon} and~\ref{lem:extend}.

\begin{proof}[Proof of Theorem~\ref{thm:forest}]
Let $F$ be a forest, let $r \ge 3$, and let $H$ be a graph with $\chi(H) = r$
and $F \in \cM(H)$. Observe that we have $H \subset F + K_{r-2}(|H|)$. Let
$\gamma > 0$, and let~$G$ be an $H$-free graph with $$\delta(G) \,\ge\, \left(
\frac{2r-5}{2r-3} + 2\gamma \right) n,$$ where $n = |G|$. We shall show that
$\chi(G)$ is bounded above by some constant $C = C(H,\gamma)$.

The first step is to apply the minimum degree form of Szemer\'edi's Regularity
Lemma to $G$, with 
\begin{equation}\label{eq:forest:epsd}
  d := \frac{\gamma}{2},\qquad k_0 :=  r^2\qquad\text{and}\qquad
  \eps := \min\bigg\{\frac{\gamma}{2}, \,\frac{d^2}{2d+2|H|}\bigg\}\,.
\end{equation}

We obtain a partition $V(G) = V_0 \cup V_1 \cup
\ldots \cup V_k$, where $k_0 \le k \le k_1 = k_1(\eps,d,k_0)$, with an
$(\eps,d)$-reduced graph~$R$ such that $$\delta(R) \; \geByRef{eq:forest:epsd}
\; \left( \frac{2r-5}{2r-3} + \gamma \right) k.$$ We now partition the vertices
of $V(G)$ depending upon the collection of the sets $V_i$ to which they send
`many' edges. More precisely, define $V(G) = \bigcup_{I\subset[k]} X_I$ by
setting
\begin{equation*}
X_I:= \Big\{ v\in V(G) \,\colon\,   i \in I \Leftrightarrow |N(v) \cap V_i| \ge d|V_i| \Big\},
\end{equation*}
for each $I \subset [k]$. We claim that $\chi\big(G[X_I]\big) \le
\max\{C_1,C_2+1\}$ for all~$I\subset[k]$, where $C_1 = C_1(H,\gamma)$ and $C_2
= C_2(H,\gamma)$ are constants defined below. Since the $X_I$ form a partition,
this implies that $\chi(G)\le2^k\max\{C_1,C_2+1\}\le 2^{k_1}\max\{C_1,C_2+1\}=
C(H,\gamma)$ as desired.

In order to establish this claim we distinguish two cases.

\medskip
 \noindent \textbf{Case 1}: $|I| \, \ge \, \left( \ds\frac{2r-4}{2r-3} \right)
 k$. \medskip
 
\noindent In this case we shall show that $|X_I|\le C_1$ (where $C_1$ is a constant defined
below, and independent of $n$), and thus trivially $\chi\big(G[X_I]\big)\le C_1$.
We first claim that $R[I]$ contains a copy of $K_{r-1}$. Indeed, by our minimum
degree condition on $R$, we have $$\delta\big( R[I] \big) \; \ge \; \delta(R) -
\big( k - |I| \big) \; \ge \; |I| - \left( \frac{2}{2r-3} - \gamma \right) k \;
\ge \; \left( \frac{r-3}{r-2} + \gamma \right) |I|.$$
Thus, by Tur\'an's Theorem (or simply by proceeding greedily), $R[I]$ contains a
copy of~$K_{r-1}$, as claimed. Let $\{W_1,\dots,W_{r-1}\} \subset
\{V_1,\ldots,V_k\}$ be the set of parts corresponding to this copy of~$K_{r-1}$.

Now let~$x \in X_I$, set~$W'_i = N(x) \cap W_i$ for each $i \in [r-1]$, and note
that $|W'_i|\ge d|W_i|$, by the definition of~$X_I$. By~Fact~\ref{prop:subpair},
each pair $(W'_i,W'_j)$, $i\neq j$, is $(\eps/d,d-\eps)$-regular. By the
Counting Lemma and~\eqref{eq:forest:epsd}, it follows that $G[N(x)]$ contains at least $$\alpha_1
n^{(r-1)|H|}$$ copies of $K_{r-1}(|H|)$, for some $\alpha_1 = \alpha_1(H,\gamma) > 0$.
%where \[\alpha_1=\frac{1}{(r-1)!}(d-\eps-\tfrac{\eps}{d}|H|)^{\binom{r-1}{2}|H|^2} \Big(\frac{1}{k_1}\Big)^{(r-1)|H|}\gByRef{eq:forest:epsd}0\,.\]
% Sorry, I think this line is way too complicated to be worthwhile! 

Thus, by Lemma~\ref{pigeon}\ref{pigeon:b} (applied with
$\alpha = \alpha_1$, $s = r - 2$ and $X = X_I$), we have either $H \subset G$, a
contradiction, or $|X_I| \le |H|/\alpha_1=C_1(H,\gamma)$, as claimed.
  
\medskip \noindent \textbf{Case 2}: $|I| \, \le \, \left( \ds\frac{2r-4}{2r-3}
\right) k$. \medskip
  
\noindent In this case we shall show that  $G[X_I]$ is $C_2$-degenerate (where
$C_2$ is defined below and independent of $n$), using Lemma~\ref{lem:extend}. It
follows that $\chi\big(G[X_I]\big)\le C_2+1$. First, we shall show that
every edge of $G[X_I]$ is contained in at least $\gamma^{r-2} n^{r-2}$ copies of
$K_r$.

Let $V_I:=\bigcup_{i \in I} V_i$ denote the set of vertices in clusters
corresponding to~$I$, and let $xy \in E(G[X_I])$. By the definition of~$X_I$,
$x$ and $y$ each have at most $(d+\eps)n$ neighbours outside~$V_I$, and
thus, since $d+\eps<\gamma$, they each have at least $\big(\frac{2r-5}{2r-3} +
\gamma \big)n$ neighbours in~$V_I$.  It follows that they have at least
$2\big(\frac{2r-5}{2r-3}+\gamma\big)-|V_I|$ common neighbours in $V_I$. Finally,
since $|I|\le\big(\frac{2r-4}{2r-3}\big)k$, we have
$|V_I|\le\big(\frac{2r-4}{2r-3}\big)n$, and thus $x$ and $y$ have at least
\[\left( \frac{2r-6}{2r-3} + \gamma \right) n\] common neighbours in~$V_I$.

Now, apply Lemma~\ref{pigeon}\ref{pigeon:a} with $\alpha = \gamma^{r-2}$, $\delta = \frac{2r-5}{2r-3} +
\gamma$, $s = r - 2$, and $X = N(x) \cap N(y)$. We have $$\frac{2r-6}{2r-3} +
\gamma \; = \;  (r-2) \gamma + (r-3)\left( 1 - \frac{2r-5}{2r-3} - \gamma
\right)\,,$$ and so $G[N(x) \cap N(y)]$ contains at least $\gamma^{r-2} n^{r-2}$
copies of~$K_{r-2}$, i.e., every edge of $G[X_I]$ is contained in $\gamma^{r-2}
n^{r-2}$ copies of~$K_r$. Hence, by Lemma~\ref{lem:extend} (applied with
$\alpha=\gamma^{r-2}$, $s = r - 2$ and $t=|H|$), there exists
$\alpha'=\alpha'(\gamma^{r-2},r-2,|H|)>0$ such that $G[X_I]$ is
$2|H|/\alpha'=C_2$-degenerate, and so $\chi\big(G[X_I]\big) \le C_2+ 1$, as
required.
\end{proof}

\section{Borsuk-Hajnal graphs}\label{borsuksec}

In this section we shall describe the constructions (based on those
in~\cite{LT}) which provide the lower bounds on $\delta_\chi(H)$ in
Theorem~\ref{mainthm}. 
One of the main building blocks in these constructions is a class of graphs
which also mark the first (and most famous) application of algebraic topology
in combinatorics: the 
\emph{Kneser graphs} $\Kn(n,k)$, which are defined as follows. Given $n,k \in
\NATS$, let $\Kn(n,k)$ have vertex set $\binom{[n]}{k}$, the family of
$k$-vertex subsets of $[n]$, and let $\{S,T\}$ be an edge if and only if $S$
and $T$ are disjoint. (For example, $\Kn(5,2)$ is the well-known Petersen graph.)
These graphs were first studied by Kneser~\cite{Kneser}, who conjectured
that $\chi\big( \Kn(n,k) \big) = n - 2k + 2$ for every $n$ and $k$. This
problem stood open for 23 years, until it was solved by Lov\'asz~\cite{Lovasz},
whose proof led eventually to a new area, known as Topological
Combinatorics (see~\cite{Mat03}, for example). 
% Shorter proofs of Kneser's
% Conjecture (also using Borsuk-Ulam) were given by B\'ar\'any~\cite{Bara}
% and Greene~\cite{Greene02}, and a combinatorial proof was found by
% Matou\v{s}ek~\cite{Mat04}.

Hajnal (see~\cite{ES}) used the Kneser graphs in order to give the first
examples of dense triangle-free graphs with high chromatic number. Given
$k,\ell,m \in \NATS$ such that $2m + k$ divides $\ell$, let the
\emph{Hajnal graph}, denoted $H(k,\ell,m)$, be the graph obtained as
follows: first take a copy of $\Kn(2m+k,m)$, and a complete bipartite graph
$K_{2\ell,\ell}$, with vertex set $A \cup B$, where $|A| = 2\ell$, and $|B|
= \ell$; next partition $A$ into $2m + k$ equally sized pieces
$A_1,\ldots,A_{2m+k}$; finally, add an edge between $S \in V\big(\Kn(2m+k,m)\big)$ and
$y \in A_j$ whenever $j \in S$.

The following theorem, which appeared in~\cite{ES}, implies that $\delta_\chi(K_3) \ge 1/3$.

\begin{theorem}[Hajnal, 1973]\label{Hajnal} 
For all $\nu>0$ and $k \in \NATS$ there exist integers $m$ and~$\ell_0$ 
such that, for every~$\ell \ge \ell_0$, the Hajnal graph $G=H(k,\ell,m)$ satisfies $v(G)=3\ell+\binom{2m+k}{m}$, 
$\chi(G)\ge k+2$ and $\delta(G)\ge(\frac13-\nu)v(G)$, and is triangle-free.
\end{theorem}

In order to generalise Theorem~\ref{Hajnal} from triangles to arbitrary
3-chromatic graphs which are not near-acyclic, {\L}uczak and
Thomass\'e~\cite{LT} defined the so-called Borsuk-Hajnal graphs. We shall
next describe their construction.

The \emph{Borsuk graph} $\Bor(k,\eps)$ has vertex set $S^k$, the $k$-dimensional
unit sphere, and edge set $\{xy\colon \ang(x,y) \ge \pi - \eps\}$, where
$\ang(x,y)$ denotes the angle between the vectors~$x$ and~$y$. It follows from
the Borsuk-Ulam Theorem (see~\cite{Mat03}, for example) that
$\chi\big(\Bor(k,\eps)\big) \ge k + 2$ for any $\eps > 0$.

In order to construct Borsuk-Hajnal graphs from Borsuk graphs, we also need
the following theorem, which follows easily from a result of of
Ne\v{s}et\v{r}il and Zhu~\cite{NesZhu}.
% of Hell and Ne\v{s}et\v{r}il~\cite{HelNes}. 

\begin{theorem}\label{HellNes}
  Given $\ell \in \NATS$ and a graph $G$, there exists a graph $G'$ with
  girth at least $\ell$, such that $\chi(G') = \chi(G)$, and such that there
  exists a homomorphism $\phi$ from $G'$ to $G$.
\end{theorem}

Now, given $k \in \NATS$, a set $W \subset S^k$ with $|W|$ even, and $\eps,\delta > 0$, we define the Borsuk-Hajnal graph, $\BH=\BH(W;k,\eps,\delta)$, as follows. 

First, let $B = \Bor(k,\eps)$ be the Borsuk graph, and let $U \subset S^k = V(B)$ be a finite set, with $U$ chosen 
such that $\chi(B[U]) = k + 2$. (This is possible by the de Bruijn-Erd\H{o}s Theorem~\cite{BruErd51}, 
which states that every infinite graph with chromatic number~$k'$ has a finite
subgraph with chromatic number~$k'$.) Let $B'$ denote the graph given by Theorem~\ref{HellNes}, applied with $G=B[U]$ 
and $\ell=k$, let $\phi$ be the corresponding homomorphism from~$B'$ to~$B[U]$, 
and let~$U'$ be the vertex set of~$B'$. 

Let $X$ be a set of size $|W|/2$, and recall that $K[W,X]$ denotes the edge set of the complete bipartite graph with parts $W$and $X$.

\begin{definition}[The Borsuk-Hajnal graph] \label{def:BH}
Define $\BH=\BH(W;k,\eps,\delta)$ to be the graph on vertex set $U' \dcup W
\dcup X$, where $U'$, $W$ and $X$ are pairwise disjoint and
as described above, with the following edges: \[E(B') \cup K[W,X] \cup \Big\{
\{u,w \} : u \in U', \, w \in W \textup{ and } \ang\big(\phi(u),w\big) <
\frac{\pi}{2} - \delta \Big\}\,.\]
\end{definition}

Observe that~$\chi(\BH) \ge \chi\big(B[U]\big) > k$.

\begin{theorem}[{\L}uczak and Thomass\'e~\cite{LT}]\label{LTborsuk}  
  For every $k \in \NATS$ and $\nu>0$, there exist $\eps,\delta > 0$ and $W \subset S^k$, 
  such that, setting $\BH=\BH(W;k,\eps,\delta)$, we have
  $$\chi(\BH) \,\ge\, k \qquad \text{ and } \qquad \delta(\BH) \,\ge\, \left( \frac13 - \nu \right)v(\BH).$$
  Moreover every subgraph $H \subset \BH$ with $v(H) < k$ and $\chi(H) = 3$ is near-acyclic.

  Hence, for any~$H$ with $\chi(H)=3$ which is not near-acyclic, we have $\delta_\chi(H)\ge 1/3$.
\end{theorem}

We shall generalise the {\L}uczak-Thomass\'e construction further, as follows, to give
our claimed lower bound on $\delta_\chi(H)$ for $r$-chromatic $H$ which are
not $r$-near-acyclic. 

\begin{definition}[The $r$-Borsuk-Hajnal graph] \label{def:BHr}
Define $\BH_r(W;k,\eps,\delta)$ to be the graph obtained from the Borsuk-Hajnal graph
$\BH = \BH(W;k,\eps,\delta)$ by adding $r-3$ independent sets
$Y_1,\dots,Y_{r-3}$ of size $|Y_1| = \ldots = |Y_{r-3}| = |W|$, and the
following edges: 
$$\bigcup_{1 \le i < j \le r-3}  K[Y_i,Y_j] \,\cup\, \bigcup_{i = 1}^{r-3} K[Y_i,V(\BH)].$$
That is, we add the complete $(r-2)$-partite graph on $V(\BH) \cup Y_1 \cup \ldots \cup Y_{r-3}$.
\end{definition}

The following result extends Theorem~\ref{LTborsuk} to arbitrary $r \ge 3$.

\begin{theorem}\label{thm:borsuk}
For every $r \ge 3$, $k \in \NATS$ and $\nu>0$, there exist $\eps,\delta > 0$ and $W \subset S^k$, 
  such that, setting $\BH_r=\BH_r(W;k,\eps,\delta)$, we have
  $$\chi(\BH_r) \,\ge\, k \qquad \text{ and } \qquad \delta(\BH_r) \,\ge\, \left( \frac{2r-5}{2r-3} - \nu \right)v(\BH_r).$$
  Moreover every subgraph $H \subset \BH_r$ with $v(H) < k$ and $\chi(H) = r$ is $r$-near-acyclic.

  Hence, for any~$H$ with $\chi(H)=r$ which is not $r$-near-acyclic, we have $\delta_\chi(H)\ge \frac{2r-5}{2r-3}$.
\end{theorem}

Theorem~\ref{thm:borsuk} follows easily from \L uczak and Thomass\'e's
argument for Theorem~\ref{LTborsuk}; for completeness, we shall provide a
proof here.

\begin{proof}[Proof of Theorem~\ref{thm:borsuk}]
As noted above, we have $\chi\big( \BH_r(W;k,\eps,\delta) \big) > k$ for every choice of $W$, $\eps$ and $\delta$. For the other properties, we shall choose $W$ randomly, and $\eps,\delta > 0$ as follows.

Let~$r \ge 3$, $k \in \NATS$, and~$\nu > 0$ be arbitrary, and choose $\delta >
0$ such that the spherical cap of~$S^k$ (centred around the pole) with polar
angle $\frac{\pi}{2}-\delta$ covers a $(\frac12-\frac\nu2)$-fraction
  of~$S^k$. Set $\eps = \delta/(2k)$,
% \begin{equation}\label{eq:borsuk:u}
   $$ u_0 \,:=\, v\big(\BH_r(\emptyset;k,\eps,\delta)\big) \, = \, |U'|,$$
% \end{equation}
  and choose $w_0$ sufficiently large such that
  \begin{equation}\label{eq:borsuk:w}  % Do we really need to give this much detail about the smallest possible W?
    \exp\left(-\frac{\nu^2w_0}{4}\right)<\frac{1}{u_0}
   \quad\text{and}\quad
    \left(\frac{2r-5}{2}-\nu\right)w_0\ge
    \left(\frac{2r-5}{2r-3}-\nu\right)\left(\frac{2r-3}{2}w_0+u_0\right),
  \end{equation}
  which is possible because $(2r-3)/2 > 1$. Draw~$w_0$ points uniformly at
  random from~$S^k$, call the resulting set~$W$ and consider the graph
  $\BH_r=\BH_r(W;k,\eps,\delta)$.
  
  We show first that, with positive probability, $\BH_r$ has the desired minimum
  degree. Let $Y = Y_1 \cup \ldots \cup Y_{r-3}$, and recall that $\BH_r$ has
  vertex set $U' \dcup W \dcup X \dcup Y$ and that $|U'| = u_0$.

 \begin{claim}\label{cl:borsuk:degree}
   With positive probability the following holds. For every $v\in V(\BH_r)$,
   \begin{equation*}
     \deg_{\BH_r}(v) \,\ge\, \left( \frac{2r-5}{2} - \nu \right) |W|.
   \end{equation*}
 \end{claim}
 
 \begin{claimproof}[Proof of Claim~\ref{cl:borsuk:degree}]
  For $v \in W \cup X \cup Y$, it is easy to check that $\deg_{\BH_r}(v) \ge \left( \frac{2r-5}{2} \right)|W|$.
  Moreover, if $v \in U'$ then $Y \subset N(v)$ and $|Y| = (r-3)|W|$. Thus it will suffice to show 
  that the following event~$\sigma$ holds with positive
   probability: For every $v \in U'$ we have $\deg_{\BH_r}(v,W)\ge (\tfrac12-\nu)|W|$. 
   
   To this end observe that, for a given $v \in U'$, the value of $\deg_{\BH_r}(v,W)$ is a random variable~$B$
   with distribution $\Bin\big( |W|,\frac12-\frac\nu2 \big)$. This follows because~$W$ was chosen
   uniformly at random from~$S^k$, by our choice of $\delta$, and since by
   Definition~\ref{def:BH}, $v$ is adjacent to $w\in W$ if and only if
   $\ang(\phi(v),w)\le \frac\pi2 - \delta$. Thus, by Chernoff's inequality (see, e.g., 
   \cite[Chapter~2]{JaLuRu:Book}),
  \begin{equation*}
     \Prob\Big(\deg_{\BH_r}(v,W) < \big( \tfrac12 - \nu \big)|W|\Big) \,\le\, \exp\Big( -\frac{\nu^2|W|}{4}\Big) 
     \,\lByRef{eq:borsuk:w}\, \frac{1}{|U'|}.
   \end{equation*}
By the union bound, the event~$\sigma$ holds with positive probability, as required.
\end{claimproof}

Using~\eqref{eq:borsuk:w}, we have 
$$\left( \frac{2r-5}{2} - \nu \right) |W| \, \ge \, \left( \frac{2r-5}{2r-3} - \nu \right) v(\BH_r),$$
and so the desired lower bound on $\delta(\BH_r)$ follows immediately from the claim.

\smallskip

  Finally, let us show that every subgraph~$H \subset \BH_r$ with $v(H) < k$ and $\chi(H) = r$ is
  $r$-near-acyclic. We begin by showing that $H' := H[U'\cup W \cup X]$ is near-acyclic. 
  
  Observe first that $H'[W]$ is independent, and recall (from
  Definition~\ref{def:BH}) that $\BH_r[U']$ has girth at least $k >
  v(H)$. Thus $H'[U' \cup X]$ is a forest, since all of its edges are
  contained in $U'$. It therefore suffices to prove the following claim.
    
   \begin{claim}\label{cl:borsuk:path}
   Every odd cycle in $H'$ contains at least two vertices of $W$.
  \end{claim}
  
  \begin{claimproof}[Proof of Claim~\ref{cl:borsuk:path}]
    Let~$C$ be an odd cycle in~$H'$. (Hence $v(C)<k$.) If $V(C)\cap X \neq
    \emptyset$ then $|V(C)\cap W|\ge 2$ since $e(U',X)=0$ and~$X$ is
    independent. Thus we may assume that $V(C) \cap X =
    \emptyset$. Similarly, since $H'[U']$ is a forest, we must have $|V(C)
    \cap W| \ge 1$.
  
    Let $P = v_1\dots v_p$ be a path in~$U'$ with $p < k$ and~$p$ even.
    Recall that $\phi(v_1),\dots,\phi(v_p)$ are vectors from $S^k$ such
    that $\ang\big(\phi(v_i),\phi(v_{i+1})\big)\ge\pi-\eps$ for all
    $i\in[p-1]$.  We shall show that $N_{\BH_r}(v_1)\cap N_{\BH_r}(v_p)\cap
    W = \emptyset$, i.e., that $\phi(v_1)$ and~$\phi(v_p)$ do not lie in a
    common spherical cap with angle $\frac\pi2 - \delta$.  Indeed, we have
    $\ang\big(\phi(v_1),\phi(v_3)\big)\le 2\eps$, and, in general,
    $\ang\big(\phi(v_1),\phi(v_{2j})\big)\ge\pi-2j\eps$ for every
    $j\in[p/2]$.  Hence $\ang\big(\phi(v_1),\phi(v_p)\big)\ge\pi-k\eps >
    2(\frac\pi2-\delta)$, and so, by Definition~\ref{def:BH}, $v_1$ and
    $v_p$ do not have a common neighbour in~$W$, as required.

    This implies that $V(C) \cap U'$ cannot be a path on $v(C)-1$ vertices and
    thus we conclude $|V(C) \cap W| \ge 2$.
 \end{claimproof}

  Finally, note that as $H[Y_i]$ is an independent set for each $i \in [r-3]$, and $H'$ is obtained by removing these sets, $H$ is indeed $r$-near-acyclic, as required.
   \end{proof}

\section{Zykov graphs}\label{Zykovsec}

In this section we shall prove a key result on Zykov graphs (see
Definition~\ref{def:Zykov} and Proposition~\ref{lem:rich-H}, below), which will
be an important tool in our proof of Theorem~\ref{mainthm}. 
Let $G'$ be a graph
with connected components $C_1,\ldots,C_m$, and let $G$ be the graph obtained
from $G'$ by adding, for each $m$-tuple $\bfu = (u_1,\ldots,u_m) \in C_1 \times
\dots \times C_m$, a vertex $v_\bfu$ adjacent to each~$u_j$. This construction
was introduced by Zykov~\cite{Zykov} in order to obtain triangle-free graphs
with high chromatic number.

We shall use the following slight modification of Zykov's construction. Recall
that to blow up a vertex $v \in V(G)$ to size~$t$ means to replace $v$ by an
independent set of size~$t$, and replace each edge containing $v$ by a complete bipartite
graph, and that $K(v,X)$ denotes the set of pairs $\{ vx : x \in X\}$.

\begin{definition}[Modified Zykov graphs]\label{def:Zykov}
Let $T_1,\ldots,T_\ell$ be (disjoint) trees, and let $T_j$ have bipartition $A_j
\dcup B_j$. We define~$Z_\ell(T_1,\ldots,T_\ell)$ to be the graph on
vertex set $$V\big( Z_\ell(T_1,\ldots,T_\ell) \big) := \bigg( \bigcup_{j \in
[\ell]} A_j \cup B_j \bigg) \cup \big\{ u_I \colon I \subset [\ell] \big\}$$ and
with edge set
\begin{equation*}
E\big( Z_\ell(T_1,\ldots,T_\ell) \big) \, := \, \bigcup_{j = 1}^\ell
\Bigg(E(T_j) \cup \bigcup_{j \in I \subset [\ell]} K\big( u_I,A_j \big) \cup
\bigcup_{j \not\in I \subset [\ell]} K\big( u_I,B_j \big) \Bigg).
\end{equation*}
For each $r \ge 3$ and $t \in \NATS$, the \emph{modified Zykov graph}
$Z_\ell^{r,t}(T_1,\ldots,T_\ell)$ is the graph obtained from
$Z_\ell(T_1,\ldots,T_\ell)$ by performing the following two operations:
\begin{enumerate}[label=\abc]
  \item Add vertices $W = \{ w_1, \dots, w_{r-3} \}$, and all edges with
    an endpoint in $W$.
  \item Blow up each vertex $u_I$ with $I \subset [\ell]$ and each
    vertex $w_j$ in $W$ to a set $S_I$ or $S'_j$, respectively, of size $t$.
\end{enumerate}
Finally, we shall write $Z_\ell^{r,t}$ for the modified Zykov graph obtained
when each tree $T_i$, $i\in[\ell]$, is a single edge; that is,
$Z_\ell^{r,t} = Z_\ell^{r,t}(e_1,\ldots,e_\ell)$.
\end{definition}

Note that $Z_\ell^{r,t}$ has $(2^\ell + r - 3)t + 2\ell$ vertices, and that, in
the special case $r = 3$ and $t = 1$, the graph $Z_\ell^{r,t}$ coincides with
that obtained by Zykov's construction (described above) applied to a matching of
size $\ell$.

The following observation motivates (and follows immediately from)
Definition~\ref{def:Zykov}.

\begin{obs}\label{acy=Zyk}
Let $\chi(H) = r$. Then $H$ is $r$-near-acyclic if and only if there exist trees
$T_1,\ldots,T_\ell$ and $t \in \NATS$  such that $H$ is a subgraph of
$Z_\ell^{r,t}(T_1,\ldots,T_\ell)$.
\end{obs}

\begin{proof}
Recall that $H$ is $r$-near-acyclic if and only if there exist $r-2$ independent
sets $U_1,\ldots,U_{r-3},W$ such that $H \setminus \big(W\cup\bigcup_j U_j\big)$
is a forest $F$
whose components are trees $T_1,\ldots,T_\ell$ with the following property. For
each $i\in[\ell]$, there is no vertex of $W$ adjacent to vertices in both
partition classes of $T_i$. If $H\subset Z_\ell^{r,t}(T_1,\ldots,T_\ell)$ then we can take $W =
\bigcup_{I \subset [\ell]}S_I$ and $U_1,\ldots,U_{r-3}$ to be the sets
$S'_1,\ldots,S'_{r-3}$, and so $H$ is $r$-near-acyclic, as claimed. Conversely
if $H$ is $r$-near-acyclic, then $H \subset Z_\ell^{r,t}(T_1,\ldots,T_\ell)$,
where $t = |H|$ and $T_1, \ldots, T_\ell$ are the components of the forest $F$.
\end{proof}

It will be convenient for us to provide a compact piece of notation for the
adjacencies in $Z_\ell^{r,t}$. For this purpose, given a graph $G$ and a set $Y\subset
V(G)$, and integers $\ell,t \in \NATS$ and $r \ge 3$, define $\cG^{r,t}_\ell
(Y)$ to be the collection of functions
\[S \,:\, 2^{[\ell]} \cup [r-3] \,\to\, \binom{Y}{t}\,.\]
It is natural to think of $S$ as a family $\{S_I : I \subset [\ell]\} \cup \{
S'_j : j \in [r-3]\}$ of subsets of $Y$ of size $t$. We say that $S \in
\cG_\ell^{r,t}(Y)$ is \emph{proper} if these sets are pairwise disjoint and
$E(G)$ contains all edges $xy$ with $x\in S_I\cup S'_j$ and $y\in S'_{j'}$, whenever $j\neq j'$. We
shall write $\cF_\ell^{r,t}(Y)$ for the collection of proper functions
in $\cG_\ell^{r,t}(Y)$. The idea behind this definition is that we will later
want to consider a vertex set $Y\subset V(G)$ and a family of disjoint subsets
$\{S_I\colon I\subset[\ell]\}\cup\{S'_j\colon j\in[r-3]\}$ of size $t$ in $Y$
that we want to extend to a copy of $Z_\ell^{r,t}$.

For an ordered pair $(x,y)$ of vertices of $G$, a function
$S\in\cF_\ell^{r,t}(Y)$, and $i \in [\ell]$, we write $(x,y) \to_i S$, if
$S'_j \subset N(x,y)$ for every $j \in [r-3]$ and
\[\bigcup_{I \,:\, i \in I} S_I \subset N(x)  \qquad \textup{and} \qquad
\bigcup_{I \,:\, i \not\in I} S_I \subset N(y)\,.\] For an edge $e=xy\in
E(G)$, we write $e\to_i S$ if either $(x,y)\to_i S$ or $(y,x)\to_i S$.
Recall that $\tpl{e}{\ell}$ denotes the $\ell$-tuple $(e_1,\ldots,e_\ell)$, with
$\tpl{e}{0}$ the empty tuple. Define
\[\tpl{e}{\ell} \to S \qquad \Leftrightarrow \qquad e_i
\to_i S \quad \textup{for each $i \in [\ell]$}\,.\] 
Observe that the graph
$Z_\ell^{r,t}$ consists of a set of pairwise disjoint edges $e_1,\ldots,e_\ell$
and an $S \in \cF_\ell^{r,t}(Y)$ such that $\tpl{e}{\ell} \to S$. An advantage
of this notation is that we can write $\tpl{e}{\ell}\to S$ even if the edges in
$\tpl{e}{\ell}$ are not pairwise disjoint. This will greatly clarify our proofs.

In Section~\ref{VCsec}, we shall show how to find a well-structured set of many
copies of $Z_\ell^{r,t}$ inside a graph with high minimum degree and high chromatic
number. The following definition (in which we shall make use of the compact
notation just defined) makes the concept of `well-structured' precise. Recall
that, given $X \subset V(G)$, we write $E(X)$ for the edge set of $G[X]$, and
that if $D \subset E(G)$, then $\delta(D)$ denotes the minimum degree of the
graph $G[D]$.

\begin{definition}[$(C,\alpha)$-rich in copies of $Z_\ell^{r,t}$]\label{def:rich}
Let~$X$ and~$Y$ be disjoint vertex sets in a graph~$G$, let~$C \in \NATS$ and
$\alpha > 0$, and let $s := (2^\ell+ r - 3) t$. We say that $(X,Y)$ is
\emph{$(C,\alpha)$-rich} in copies of $Z_\ell^{r,t}$ if
\begin{align*}
& \exists \, D=D(\tpl{e}{0})\subset E(X) \; \forall  \, e_1\in D \; \exists \,
D(\tpl{e}{1})\subset E(X) \; \forall \, e_2 \in D(\tpl{e}{1}) \quad \dots \\ &
\hspace{2cm} \dots \quad \forall \, e_{\ell-1} \in D(\tpl{e}{\ell-2}) \;
\exists \, D(\tpl{e}{\ell-1}) \subset E(X) \; \forall \, e_\ell \in
D(\tpl{e}{\ell-1})
\end{align*}
the following properties hold:
\begin{enumerate}[label=\abc]
  \item\label{def:rich:a} $\delta(D), \delta\big(D(\tpl{e}{1})\big), \dots,
  \delta\big(D(\tpl{e}{\ell-1})\big) > C$, and 
  \item\label{def:rich:b} $\left|
  \big\{ S \in \cF_\ell^{r,t}(Y) \,:\, \tpl{e}{\ell} \to S \big\} \right|  \ge
  \alpha |Y|^s$.
\end{enumerate}
If $(X,Y)$ is $(C,\alpha)$-rich in copies of $Z_\ell^{r,t}$, then, for each $q \in [\ell]$, define
\begin{equation}\label{def:Dq}
\cD_q(X,Y) \, := \, \Big\{ \tpl{e}{q} \in E(X)^q \,:\, e_{j} \in D(\tpl{e}{j-1})
\textup{ for each } j \in [q] \Big\},
\end{equation}
where $D(\tpl{e}{0}) := D$.
\end{definition}

The aim of this section is to prove the following proposition, which says that
if some pair $(X,Y)$ in $G$ is $(C,\alpha)$-rich in copies of $Z_\ell^{r,t}$
(where $\alpha > 0$ and $C$ is sufficiently large), then for any `small'
$T_1,\ldots,T_\ell$ we have $Z_\ell^{r,t}(T_1,\ldots,T_\ell) \subset G$, and
hence (by Observation~\ref{acy=Zyk}) $G$ is not $H$-free for any $r$-near-acyclic
graph~$H$.% with $\chi(H)=r$.

\begin{prop}\label{lem:rich-H}
Let~$G$ be a graph, and let $X$ and $Y$ be disjoint subsets of its vertices. Let
$r,\ell, t \in \NATS$, with $r \ge 3$, and let $\alpha > 0$. Let
$T_1,\ldots,T_\ell$ be trees, and set $C := 2^{\ell+3} \alpha^{-1}
\sum_{i=1}^{\ell} |T_i|$.
If $(X,Y)$ is $(C,\alpha)$-rich in copies of $Z_\ell^{r,t}$, then
$Z_\ell^{r,t}(T_1,\ldots,T_\ell) \subset G$.
\end{prop}

The proof of Proposition~\ref{lem:rich-H} uses a double counting argument
and proceeds by induction. We shall find a set of functions $\cS
\subset \cF^{r,t}_\ell(Y)$ such that, for each $S \in \cS$, we can
construct (one by one) a collection of subgraphs $E_1,\ldots,E_\ell$ of
$G[X]$ with the following properties: each subgraph has large average
degree, and for \emph{any} choice $e_1 \in E_1,\ldots,e_\ell \in E_\ell$,
we have $\tpl{e}{\ell} \to S$. Recall that this simply says that $e_i \to_i
S$ for every $e_i \in E_i$.

Let the graph $G$, disjoint subsets $X,Y \subset V(G)$, constants $\alpha > 0$
and $r, \ell, t \in \NATS$ with $r \ge 3$, and trees $T_1,\ldots,T_\ell$ be
fixed for the rest of the section. Set $C := 2^{\ell+3} \alpha^{-1}
\sum_{i=1}^{\ell} |T_i|$ and $s := (2^\ell + r - 3)t$, and let $0 \le q \le
\ell$. For our induction hypothesis we use the following definition.

\begin{definition}[Good function, $(C,\alpha)$-dense]\label{def:good}
  A function $S \in \cF_\ell^{r,t}(Y)$ is \emph{$(r,\ell,t,C,\alpha)$-good}
  for a tuple $\tpl{e}{q}$ and $(X,Y)$ if there exist
  sets $$E_{q+1},\ldots,E_\ell \subset E(X), \quad \textup{with} \quad
  \ol{d}(E_j) \ge 2^{-\ell} \alpha C \quad \textup{ for each $q + 1 \le j
    \le \ell$,}$$ such that for every $e_{q+1} \in E_{q+1}, \ldots, e_\ell
  \in E_\ell$, we have $\tpl{e}{\ell} \to S$.

  When the constants $(r,\ell,t,C,\alpha)$ and the sets $(X,Y)$ are clear
  from the context, we shall omit them. We shall abbreviate
  `$(r,\ell,t,C,\alpha)$-good for $\tpl{e}{0}$ and $(X,Y)$' to
  `$(r,\ell,t,C,\alpha)$-good for $(X,Y)$'.

  The pair $(X,Y)$ is \emph{$(C,\alpha)$-dense} in copies of $Z_\ell^{r,t}$
  if there exist at least $2^{-\ell} \alpha |Y|^s$ families $S \in \cF(Y)$
  which are $(r,\ell,t,C,\alpha)$-good for $(X,Y)$. %$\tpl{e}{0}$ and $(X,Y)$.
\end{definition}

The next lemma constitutes the inductive argument in the proof of
Proposition~\ref{lem:rich-H}. The final assertion we shall also need in
Section~\ref{proofsec}.

\begin{lemma}\label{q-induc}
  For any $0\le q\le\ell$, if $(X,Y)$ is $(C,\alpha)$-rich in copies of
  $Z_\ell^{r,t}$, then
  \[\left| \big\{ S \in \cF_\ell^{r,t}(Y) \,:\, S \text{ is
      good for } \tpl{e}{q} \big\} \right| \, \ge \, 2^{q-\ell} \alpha
  |Y|^s\] for every $\tpl{e}{q} \in \cD_q(X,Y)$.

  In particular, if $(X,Y)$ is $(C,\alpha)$-rich in copies of
  $Z_\ell^{r,t}$, then $(X,Y)$ is $(C,\alpha)$-dense in copies of
  $Z_\ell^{r,t}$.
\end{lemma}

\begin{proof}
  The proof is by induction on $\ell - q$. The base case, $q = \ell$,
  follows immediately from the definition of $(C,\alpha)$-rich. Indeed, a
  function $S$ is good for $\tpl{e}{\ell}$ if and only if $\tpl{e}{\ell}
  \to S$, and by Property~\ref{def:rich:b} in Definition~\ref{def:rich}, we
  have $\left| \big\{ S \in \cF_\ell^{r,t}(Y) \,:\, \tpl{e}{\ell} \to S
    \big\} \right| \ge \alpha |Y|^s$.

  So let $0 \le q < \ell$, assume that the lemma holds for $q + 1$, and let
  $\tpl{e}{q} \in \cD_q(X,Y)$. Set $\beta=\alpha 2^{q-\ell}$. By
  Definition~\ref{def:rich}, there exists a set $D(\tpl{e}{q}) \subset
  E(X)$ with $\delta\big( D(\tpl{e}{q}) \big) > C$, such that $\tpl{e}{q+1}
  \in \cD_{q+1}(X,Y)$ for every $e_{q+1} \in D(\tpl{e}{q})$. Thus, by the
  induction hypothesis,
  \begin{equation}\label{eq:goodS}
    \forall e_{q+1} \in D(\tpl{e}{q})\text{ at least } 2 \beta
    |Y|^s\text{ functions }S \in \cF_\ell^{r,t}(Y)\text{ are good for }
    \tpl{e}{q+1}\,,
  \end{equation}
  that is, for each such $S$ there exist sets \[E_{q+2},\ldots,E_\ell
  \subset E(X), \quad \textup{with} \quad \ol{d}(E_j) \ge 2^{-\ell} \alpha
  C \quad \textup{ for each $q + 2 \le j \le \ell$,}\] such that for every
  $e_{q+2} \in E_{q+2}, \ldots, e_\ell \in E_\ell$, we have $\tpl{e}{\ell}
  \to S$.
  It is crucial to observe that given $S$, if the edge sets
  $E_{q+2},\ldots,E_\ell$ have this last property for \emph{some}
  $e_{q+1}\in D(\tpl{e}{q})$ with $\tpl{e}{q+1}\to S$, then
  $E_{q+2},\ldots,E_{\ell}$ have this property for \emph{all} $e_{q+1}\in
  D(\tpl{e}{q})$ with $\tpl{e}{q+1}\to S$.

  The $S \in \cF_\ell^{r,t}(Y)$ that will be good for $\tpl{e}{q}$ are
  those which are good for many $\tpl{e}{q+1}$. More precisely, for each $S
  \in \cF_\ell^{r,t}(Y)$, let $$W_S \, := \, \big\{ e_{q+1} \in
  D(\tpl{e}{q}) \,:\, S \textup{ is good for } \tpl{e}{q+1} \big\},$$ and
  let $Z = \big\{ S \in \cF_\ell^{r,t}(Y) \,:\, |W_S| \ge \beta
  |D(\tpl{e}{q})| \big\}$. By~\eqref{eq:goodS} the number of pairs
  $(e_{q+1},S)$ with $e_{q+1}$ in $W_S$ is at least $|D(\tpl{e}{q})| \cdot
  2 \beta |Y|^s$. On the other hand, for every $S\in Z$ there are at most
  $|D(\tpl{e}{q})|$ pairs $(e_{q+1},S)$ with $e_{q+1}$ in $W_S$, and for
  every $S\in\cF_\ell^{r,t}(Y)\setminus Z$, there are (by definition of
  $Z$) at most $\beta|D(\tpl{e}{q})|$ such pairs. Putting these together,
  we obtain
  \[|D(\tpl{e}{q})|\cdot 2\beta |Y|^s\le |D(\tpl{e}{q})||Z|+\beta
  |D(\tpl{e}{q})||Y|^s\]
  and hence $|Z| \ge \beta |Y|^s$.

  We claim that every $S \in Z$ is good for $\tpl{e}{q}$. Indeed, fix $S\in
  Z$.  Set $E_{q+1} = W_S$, and let $E_{q+2},\ldots,E_\ell$ be the sets
  defined above (for any, and thus all, $e_{q+1}\in E_{q+1}$), i.e., those
  obtained by the induction hypothesis. Since $S \in Z$ we have $|W_S| \ge
  \beta |D(\tpl{e}{q})|$, so it follows from $\delta\big( D(\tpl{e}{q})
  \big) > C$ that $\ol{d}(E_{q+1}) \ge \beta C \ge 2^{-\ell} \alpha % It's equal only when q=0!
  C$. Since $S$ is good for $\tpl{e}{q+1}$ for every $e_{q+1} \in E_{q+1}$,
  we have $e_i \to_i S$ for every $1 \le i \le q+1$, and by the induction
  hypothesis, we have $e_i \to_i S$ for every $e_i \in E_i$ and every $q+2
  \le i \le \ell$.  Thus $\tpl{e}{\ell} \to S$ for every such
  $\tpl{e}{\ell}$, as required. Since $|Z| \ge \beta |Y|^s$, this completes
  the induction step, and hence the proof of the lemma.
\end{proof}

Lemma~\ref{q-induc} shows that richness in copies of $Z_\ell^{r,t}$
implies denseness in copies of $Z_\ell^{r,t}$.  Observe that if $(X,Y)$ is
dense in copies of $Z_\ell^{r,t}$, then in particular there is a function
$S\in\cF_\ell^{r,t}(Y)$ which is good for $(X,Y)$. The next lemma now shows
that in this case we have $Z_\ell^{r,t}(T_1,\ldots,T_\ell)\subset G$.

\begin{lemma}\label{lem:denseembed} 
Let $X$ and $Y$ be disjoint vertex sets
  in $G$. Given $r,\ell,t\in\NATS$, $\alpha>0$, and trees $T_1,\ldots,T_\ell$,
  if $C\ge 2^{\ell+3}\alpha^{-1}\sum_{i=1}^\ell |T_i|$ and
  $S\in\cF_\ell^{r,t}(Y)$ is $(r,\ell,t,C,\alpha)$-good for $(X,Y)$, then
  $Z_\ell^{r,t}(T_1,\ldots,T_\ell)\subset G$.
\end{lemma}

\begin{proof} Let $S$ be $(r,\ell,t,C,\alpha)$-good for $(X,Y)$. Then there
  exist sets
  \[E_{1},\ldots,E_\ell \subset E(X), \quad \textup{with}  \quad \ol{d}(E_j) \ge
  2^{-\ell} \alpha C \quad \textup{ for each $1 \le j \le \ell$,}\]
  such that for every $e_{1} \in E_{1}, \ldots, e_\ell \in E_\ell$, we have
  $\tpl{e}{\ell} \to S$.
  
  For each $j \in [\ell]$ and each edge $e \in E_j$, let $e=xy$ be such that
  $(x,y) \to_j S$, and orient the edge~$e$ from~$x$ to~$y$. Recall that $C \ge
  2^{\ell+3} \alpha^{-1} \sum_{i=1}^{\ell} |T_i|$, and so $\ol{d}(E_j) \ge 8 \sum_{i=1}^{\ell} |T_i|$ for
  each $j \in [\ell]$. For each $j\in[\ell]$, by choosing a maximal bipartite
  subgraph of $E_j$, and then removing at most half the edges, we can find a set
  $E'_j\subset E_j$ such that,
  \begin{enumerate}[label=\abc]
    \item $E'_j$ is bipartite, with bipartition $(A_j,B_j)$,
    \item every edge $e \in E'_j$ is oriented from $A_j$ to $B_j$, and
    \item $\smash{\ol{d}(E'_j) \ge 2 \sum_{i=1}^{\ell} |T_i|}$.
  \end{enumerate}
  Thus, by Fact~\ref{prop:forest}, there exists, for each $j \in [\ell]$, a copy
  $T'_j$ of $T_j$ in $$E'_j \,-\, \big( V(T'_1) \cup \ldots \cup V(T'_{j-1})
  \big),$$ since removing a vertex can only decrease the average degree by at most
  two. These trees, together with $S$, form a copy of
  $Z_\ell^{r,t}(T_1,\ldots,T_\ell)$ in $G$, so we are done.
\end{proof}

It is now easy to deduce Proposition~\ref{lem:rich-H} from
Lemma~\ref{q-induc} and Lemma~\ref{lem:denseembed}.

\begin{proof}[Proof of Proposition~\ref{lem:rich-H}] Let $(X,Y)$ be
  $(C,\alpha)$-rich in copies of $Z_\ell^{r,t}$, and apply Lemma~\ref{q-induc} to
  $(X,Y)$ with $q = 0$. Note that $ \cD_0(X,Y) = \{\tpl{e}{0} \}$ consists of the
  tuple of length zero, and let $\cS = \big\{ S \in \cF_\ell^{r,t}(Y) \,:\,
  S \textup{ is good for } \tpl{e}{0} \big\}$. Then $|\cS| \ge \alpha 2^{-\ell} |Y|^s$, and so
  in particular $\cS$ is non-empty. Let $S\in\cS$, and apply
  Lemma~\ref{lem:denseembed} to obtain a copy of
  $Z_\ell^{r,t}(T_1,\ldots,T_\ell)$ in~$G$.
\end{proof}

\section{The paired VC-dimension argument}\label{VCsec}

In this section we shall modify and extend a technique which was introduced by
{\L}uczak and Thomass\'e~\cite{LT}, and used by them to prove
Conjecture~\ref{LTconj} in the case where $H$ is near-bipartite. This technique
is based on the concept of
\emph{paired VC-dimension}, which generalises the well-known
Vapnik-\v{C}ervonenkis dimension of a set-system (see~\cite{Sauer,VC}).
% , or the applications in~\cite{ABBM09,ABKKW}
We shall not state our proof in the
abstract setting of paired VC-dimension, which is more general than that which
we shall require, but we refer the interested reader to~\cite{LT} for the
definition and further details.

We shall use the paired VC-dimension (or `booster tree') argument of {\L}uczak
and Thomass\'e in order to prove the following result, which may be thought of
as a `counting version' of Theorem~5 in~\cite{LT}. The case $r = 3$ of
Theorem~\ref{mainthm} will follow as an easy consequence of
Propositions~\ref{lem:VC} and~\ref{lem:rich-H} (see
Section~\ref{proofsec}).

\begin{prop}\label{lem:VC}
  For every $\ell,t \in \NATS$ and $d > 0$, there exists $\alpha > 0$ such
  that, for every $C \in \NATS$, there exists $C' \in \NATS$ such that the following holds.
  Let $G$ be a graph and let $X$ and $Y$ be disjoint subsets of $V(G)$, such that
  $|N(x) \cap Y| \ge d|Y|$ for every $x \in X$.
  
  Then either $\chi\big( G[X] \big) \le C'$, or $(X,Y)$ is $(C,\alpha)$-rich in
  copies of $Z_\ell^{3,t}$.
\end{prop}

In order to prove Proposition~\ref{lem:VC}, we shall break $X$ up into a bounded
number of suitable pieces, $X_1,\ldots,X_m$, and show that either $G[X_j]$ has
bounded chromatic number, or $(X_j,Y)$ is $(C,\alpha)$-rich in copies of $Z_\ell^{3,t}$.

Let $d(x,Y):=d\big(\{x\},Y\big) = \big|N(x)\cap Y\big|/|Y|=
e\big(\{x\},Y\big)/|Y|$ be the density of the neighbours of $x$ in $Y$. The key definition, which will allow us to choose the
sets $X_j$, is as follows.

\begin{definition}[Boosters]\label{def:boost}
  Let $G$ be a graph, let $X$ and $Y$ be disjoint subsets of $V(G)$, and let $\eps
  > 0$. We say that $x \in X$ is \emph{$\eps$-boosted} by $Y' \subset Y$ if
  $d(x,Y') \ge (1 + \eps) d(x,Y)$.
  
  Now let $C,p \in \NATS$, and let $\beta > 0$. Let $Y_1 \dcup
  \ldots \dcup Y_p$ of $Y$ be a partition of $Y$, and $X_0\dcup\ldots\dcup X_p$ be
  a partition of $X$. We say that $(X_0,\emptyset),(X_1,Y_1),\ldots,(X_p,Y_p)$ is
  a \emph{$(p,C,\eps,\beta)$-booster} of $(X,Y)$ if
  \begin{enumerate}[label=\abc]
    \item\label{def:boost:a} $G[X_0]$ is $C$-degenerate.
    \item\label{def:boost:b} Every $x \in X_j$ is $\eps$-boosted by $Y_j$, for
    each $j \in [p]$.
    \item\label{def:boost:c} $|Y_j| \ge \beta |Y|$ for every $j \in [p]$.
  \end{enumerate}
  We say that a partition $\{Y_1,\ldots,Y_p\}$ of $Y$ induces a
  $(p,C,\eps,\beta)$-booster of $(X,Y)$ if
  there exists a partition $X_0\dcup\ldots\dcup X_p$ of $X$ such that
  $(X_0,\emptyset),(X_1,Y_1),\ldots,(X_p,Y_p)$ is a
  \emph{$(p,C,\eps,\beta)$-booster} of $(X,Y)$.
\end{definition}

We remark that this is slightly different from the definition of a $p$-booster
in Section~5 of~\cite{LT}, where Condition~\ref{def:boost:a} was replaced by `$G[X_0]$ is
independent', and Condition~\ref{def:boost:c} was missing.

Using Definition~\ref{def:boost}, we can now state the second key definition.

\begin{definition}[Booster trees]\label{def:boostree}
  Let $C,p_0 \in \NATS$ and $\beta,\eps > 0$. A
  \emph{$(p_0,C,\eps,\beta)$-booster tree} for $(X,Y)$ is an oriented rooted tree $\cT$, 
  whose vertices are pairs $(X',Y')$ such that $X' \subset X$
  and $Y' \subset Y$, and all of whose edges are oriented away from the
  root, such that the following conditions hold:
  \begin{enumerate}[label=\abc]
    \item\label{def:boostree:a} The root of $\mathcal{T}$ is $(X,Y)$.
    \item\label{def:boostree:b} No vertex of $\mathcal{T}$ has more than $p_0$
    out-neighbours.
    \item\label{def:boostree:c} The out-neighbourhood of each non-leaf $(X',Y')$
    of $\mathcal{T}$ forms a $(p,C,\eps,\beta)$-booster of $(X',Y')$ for some
    $p\le p_0$.
    \item\label{def:boostree:d} If $(X',Y')$ is a leaf of $\cT$, then either
    $G[X']$ is $C$-degenerate, or there does not exist a
    $(p,C,\eps,\beta)$-booster for $(X',Y')$ for any $p \le p_0$.
  \end{enumerate}
  A vertex $(X',Y')$ of $\mathcal{T}$ is called $\emph{degenerate}$ if it is a
  leaf of $\mathcal{T}$ and $G[X']$ is $C$-degenerate.
\end{definition}

The following lemma is immediate from the definitions.

\begin{lemma}\label{exists-booster}
  Let $C,p_0 \in \NATS$, and $\beta,\eps > 0$, let $G$ be a graph, and let $X$ and $Y$ be disjoint subsets of $V(G)$. If $d(x,Y) \ge d$ for every $x\in X$, then there exists a $(p_0,C,\eps,\beta)$-booster tree
  $\cT$ for $(X,Y)$ such that $|\cT|$ is bounded as a function of $\eps$, $d$ and~$p_0$.
  
  Moreover, if  $(X',Y')$ is a non-degenerate vertex of $\mathcal{T}$, then $d(x,Y')\ge d$ for all
  $x\in X'$, and $|Y'| \ge \beta^{|\cT|}|Y|$. 
\end{lemma}

\begin{proof}
  We construct $\mathcal{T}$, with root $(X,Y)$, as follows: We simply repeatedly choose a
  $(p,C,\eps,\beta)$-booster $(X_0,\emptyset),(X_1,Y_1),\ldots,(X_p,Y_p)$
  for each leaf $(X',Y')$ of~$\mathcal{T}$ such that $G[X']$ is not
  $C$-degenerate, until this is no longer possible for any $p\le p_0$. We
  add to~$\mathcal{T}$ the vertices
  $(X_0,\emptyset),(X_1,Y_1),\ldots,(X_p,Y_p)$ as out-neighbours of
  $(X',Y')$.
  
  By the definition of `$\eps$-boosted' and the construction
  of~$\mathcal{T}$, if $(X',Y')$ is a non-degenerate vertex
  of~$\mathcal{T}$ at distance $t$ from the root, we have $d(x,Y')\ge
  (1+\eps)^t d$ for every $x\in X'$, and we have
  $|Y'|\ge\beta^t|Y|\ge\beta^{|\cT|}|Y|$. This both establishes that the
  height $h(\mathcal{T})$ of~$\cT$ is bounded in terms of~$\eps$ and~$d$,
  and that we have $d(x,Y')\ge d$ for every $x\in
  X'$. Since~$\mathcal{T}$ has no vertex of out-degree greater than~$p_0$ and
  $h(\cT)$ is bounded by a function of~$\eps$ and~$d$, it follows that
  $|\mathcal{T}|$ is bounded as a function of $\eps$, $d$ and $p_0$.
\end{proof}

The following lemma is the key step in the proof of Proposition~\ref{lem:VC}.

\begin{lemma}\label{no-p-booster}
  Let $\ell,t \in \NATS$ and $d > 0$. Let $\beta=\big(\tfrac{d}{4}\big)^\ell$
  and $\eps=\beta/2$. There exists $\alpha > 0$ such that the following holds
  for every $C \in \NATS$. Let $G$ be a graph, let $X$ and $Y$ be disjoint subsets of $V(G)$, and
  suppose that $|N(x) \cap Y| \ge d|Y|$ for every $x \in X$.
  
  If there does not exist a $(p,C,\eps,\beta)$-booster of $(X,Y)$ for any $p
  \le 2^\ell$, then $(X,Y)$ is $(C,\alpha)$-rich in copies of $Z_\ell^{3,t}$.
\end{lemma}

Lemma~\ref{no-p-booster} is proved by repeating a fairly straightforward
algorithm $\ell$ times, at each step $q\in[\ell]$ finding a set $D(\tpl{e}{q})$
as in the definition of $(C,\alpha)$-richness (see Definition~\ref{def:rich}). In order to
make the proof more transparent, we shall state a slightly more technical lemma,
which is proved by induction on $q$, and from which Lemma~\ref{no-p-booster}
follows immediately.

The following definition will simplify the statement. 
It is a slight
strengthening of the concept of $(C,\alpha)$-richness in the case $r = 3$.
Recall that $\tpl{e}{\ell}$ is just a shorthand for $(e_1,\ldots,e_\ell)$ and $\tpl{e}{0}$ is the
empty tuple. Further, recall the definitions of~$\cF_\ell^{r,t}$
and $\tpl{e}{\ell}\to S$ from Section~\ref{Zykovsec}.

\begin{definition}[$(C,\alpha,\ell)$-Zykov]\label{def:CaZykov}
  Let~$X$ and~$Y$ be disjoint vertex sets in a graph~$G$, let $C,\ell \in \NATS$
  and $\alpha > 0$. We say that $(X,Y)$ is \emph{$(C,\alpha,\ell)$-Zykov} if
  \begin{align*}
    & \exists \, D=D(\tpl{e}{0}) \subset E(X) \; \forall  \, e_1\in D \; \exists
    \, D(\tpl{e}{1})\subset E(X) \; \forall \, e_2 \in D(\tpl{e}{1}) \quad
    \dots \\ & \hspace{2cm} \dots \quad \forall \, e_{\ell-1} \in
    D(\tpl{e}{\ell-2}) \; \exists \, D(\tpl{e}{\ell-1}) \subset E(X) \; \forall
    \, e_\ell \in D(\tpl{e}{\ell-1})
  \end{align*}
  the following properties hold:
  \begin{enumerate}[label=\abc]
    \item\label{def:CaZykov:a} $\delta\big( D \big),
    \delta\big(D(\tpl{e}{1})\big), \dots, \delta\big(D(\tpl{e}{\ell-1})\big) >
    C$, and
    \item\label{def:CaZykov:b} $\exists \, S \in \cF^{3,\alpha |Y|}_\ell(Y)$
    such that $\tpl{e}{\ell} \to S$.
  \end{enumerate}
\end{definition}

We remark that the requirement \ref{def:rich:b} of
Definition~\ref{def:rich} that each tuple~$\tpl{e}{\ell}$ should extend to
\emph{many} copies of~$Z_\ell^{3,t}$ is replaced in this definition by the
requirement that~$\tpl{e}{\ell}$ should extend to \emph{one} much
bigger copy of~$Z_\ell^{3,\alpha|Y|}$. In particular, if $(X,Y)$ is
$(C,\alpha,\ell)$-Zykov, then, for any $t \in \NATS$, it is
$(C,\alpha')$-rich in copies of $Z_\ell^{3,t}$, where $\alpha' = \left(
  \frac{\alpha}{t} \right)^{2^\ell t}$ (this is shown in the proof of
Lemma~\ref{no-p-booster}).

\begin{lemma}\label{no-boost-tech}
  Let $\ell \in \NATS$ and $d > 0$. For $\beta=\left( \frac{d}{4}
  \right)^\ell$ and $\eps=\beta / 2$, the following
  holds for every $C \in \NATS$. Let $G$ be a graph, let $X$ and $Y$ be disjoint
  subsets of $V(G)$, and suppose that $|N(x) \cap Y| \ge d|Y|$ for every $x \in X$.
  
  If there does not exist a $(p,C,\eps,\beta)$-booster of $(X,Y)$ for any $p
  \le 2^\ell$, then $(X,Y)$ is $(C,\alpha_q,q)$-Zykov for
  every $q\in[\ell]$, where $\alpha_q = \left( \frac{d}{4} \right)^q$.
\end{lemma}

\begin{proof}
  Let $C,\ell \in \NATS$ and $d > 0$, and let $\beta = \left( \frac{d}{4}
  \right)^\ell$ and $\eps = \beta / 2$. Let $G$ and $X,Y$ be as described in the
  statement, and suppose that there does not exist a $(p,C,\eps,\beta)$-booster
  of $(X,Y)$ for any $p \le 2^\ell$. We proceed by induction.
  
  We begin with the base case, $q = 1$. We are required to find a set $D =
  D(\tpl{e}{0}) \subset E(X)$, with $\delta(D) > C$, such that, for every
  $e_1=xy \in D$, there exists $S(\tpl{e}{1})\in\cF_1^{3,\alpha_1|Y|}$ such
  that $\tpl{e}{1}\to S(\tpl{e}{1})$; that is, there exist disjoint sets
  $S_\emptyset(\tpl{e}{1})$ and $S_{\{1\}}(\tpl{e}{1})$ in $Y$, both of
  size $\alpha_1|Y|=\frac{d}{4} |Y|$, such that $S_{\{1\}}(\tpl{e}{1})\subset N(x)$ and
  $S_\emptyset(\tpl{e}{1})\subset N(y)$. Since there is no
  $(1,C,\eps,\beta)$-booster of $(X,Y)$, it follows that $G[X]$ is not
  $C$-degenerate, and so there exists a subgraph $G_0 \subset G[X]$ with
  $\delta(G_0) > C$. We choose $D:=E(G_0)$.  Now for each $e_1 = x y \in
  D$, let $A_\emptyset(\tpl{e}{1}): = N(y) \cap Y$ and
  $A_{\{1\}}(\tpl{e}{1}): = N(x) \cap Y$. Since
  $|A_\emptyset(\tpl{e}{1})|,|A_{\{1\}}(\tpl{e}{1})| \ge d|Y|$ by the
  assumption of the lemma, there exist disjoint sets
  $S_\emptyset(\tpl{e}{1}) \subset A_\emptyset(\tpl{e}{1})$ and
  $S_{\{1\}}(\tpl{e}{1}) \subset A_{\{1\}}(\tpl{e}{1})$ with
  $|S_\emptyset(\tpl{e}{1})|,|S_{\{1\}}(\tpl{e}{1})| = \frac{d}{4} |Y|$, as
  required.
  
  For the induction step, let $1 < q \le \ell$ and assume that the result holds
  for $q - 1$. By this induction hypothesis 
  \begin{align*}
    & \exists D(\tpl{e}{0}) \subset E(X) \forall  \, e_1\in D(\tpl{e}{0}) \; \exists \, D(\tpl{e}{1})\subset E(X)
      \; \forall \, e_2 \in D(\tpl{e}{1}) \quad \dots \\
    & \hspace{1cm} \dots \quad \forall \, e_{q-2} \in D(\tpl{e}{q-3}) \;
      \exists \, D(\tpl{e}{q-2}) \subset E(X) \; \forall \, e_{q-1} \in
      D(\tpl{e}{q-2})
  \end{align*}
  we have
  \begin{enumerate}[label=\abcstar]
    \item\label{no-boost-tech:a} $\delta\big( D(\tpl{e}{0}) \big),
    \delta\big(D(\tpl{e}{1})\big), \dots, \delta\big(D(\tpl{e}{q-2})\big) >
    C$, and
    \item\label{no-boost-tech:b} $\exists \, S(\tpl{e}{q-1}) \in
    \cF^{3,\alpha_{q-1} |Y|}_{q-1}(Y)$ such that $\tpl{e}{q-1} \to
    S(\tpl{e}{q-1})$.
  \end{enumerate}
  As in Definition~\ref{def:rich}, set
  \begin{equation*}
    \cD_q(X,Y) \, := \, \Big\{ \tpl{e}{q} \in E(X)^q \,:\, e_{j} \in
    D(\tpl{e}{j-1}) \textup{ for each } j \in [q] \Big\}\,.
  \end{equation*}
  We shall show that for every $\tpl{e}{q-1}\in \cD_{q-1}(X,Y)$, there exists
  a set of edges $D(\tpl{e}{q-1}) \subset E(X)$, with
  $\delta\big(D(\tpl{e}{q-1})\big) > C$, such that for every $e_q \in
  D(\tpl{e}{q-1})$ there exists an $S(\tpl{e}{q}) \in \cF^{3,\alpha_q
  |Y|}_{q}(Y)$ such that $\tpl{e}{q} \to S(\tpl{e}{q})$.
  
  Indeed, given $\tpl{e}{q-1}\in\cD_{q-1}(X,Y)$, by~\ref{no-boost-tech:b}
  there exists
  \[ \big\{ S_I(\tpl{e}{q-1}) \subset Y : I \subset [q-1] \big\} \, = \, S(\tpl{e}{q-1})\, \in \, \cF^{3,\alpha_{q-1} |Y|}_{q-1}(Y)\]
  with $\tpl{e}{q-1} \to S(\tpl{e}{q-1})$. In particular, note that by
  definition of $\cF^{3,\alpha_{q-1} |Y|}_{q-1}(Y)$, the sets
  $S_I(\tpl{e}{q-1})$ are disjoint, and that $|S_I(\tpl{e}{q-1}) | =
  \alpha_{q-1} |Y|$ for every $I \subset [q-1]$. Let $R = Y \setminus \bigcup_I
  S_I(\tpl{e}{q-1})$, and recall that $q \le \ell$, and that there is no
  $(p,C,\eps,\beta)$-booster of $(X,Y)$ for any $p \le 2^\ell$. Thus
  the partition $S(\tpl{e}{q-1}) \cup \{R\}$ of $Y$ does not induce a
  $(p,C,\eps,\beta)$-booster of $(X,Y)$.
  
  By our choice of $\beta$, we have $\beta\le\alpha_{q-1}$, and thus
  $|S_I(\tpl{e}{q-1})| \ge \beta|Y|$ for all $I \subset [q-1]$. Similarly, since
  $2^{q-1}\alpha_{q-1} = 2^{q-1}(d/4)^{q-1} \le 1/2 < 1 - \beta$, %\beta replaced by \alpha_{q-1}
  we have $|R|>\beta |Y|$.  
  Let $X' \subset X$  be the set of vertices which are not $\eps$-boosted by any
  of the sets $S(\tpl{e}{q-1})\cup\{R\}$. Since $S(\tpl{e}{q-1}) \cup \{R\}$
  does not induce a $(2^{q-1}+1,C,\eps,\beta)$-booster of $(X,Y)$, the graph $G[X']$ is
  not $C$-degenerate, and hence there exists a set of edges $D(\tpl{e}{q-1})
  \subset E(X')$ such that $\delta\big( D(\tpl{e}{q-1}) \big) > C$. We
  claim that this is the set we are looking for.
  
  In order to verify this, let $e_q = x y \in D(\tpl{e}{q-1})$ be
  arbitrary. Our task is to show that there exists
  $S(\tpl{e}{q})\in\cF_q^{3,\alpha_q|Y|}$ such that $\tpl{e}{q}\to
  S(\tpl{e}{q})$. Recall that $x$ and $y$ are not $\eps$-boosted by
  $S(\tpl{e}{q-1}) \cup \{R\}$. Hence $d(x,U) < (1 + \eps)d(x,Y)$ for each
  $U \in S(\tpl{e}{q-1}) \cup \{R\}$, and so, for every $I \subset [q-1]$,
  \begin{equation}\begin{split}\label{eq:boostcalc}
    e\big(x,S_I(\tpl{e}{q-1})\big)&\,=\,e(x,Y)-e\big(x,Y\setminus
    S_I(\tpl{e}{q-1})\big) \\
    &\,\ge \, e(x,Y) - \big(1 + \eps\big) \left(1 - \left(
    \tfrac{d}{4} \right)^{q-1} \right) e(x,Y)
    \,\ge \, \frac{1}{2} \left( \frac{d}{4}\right)^{q-1} e(x,Y)\,,
  \end{split}\end{equation}
  where we used $|S_I(\tpl{e}{q-1})| = \alpha_{q-1}|Y|=\left( \frac{d}{4}
  \right)^{q-1} |Y|$ for the first inequality, and
  $\eps=\tfrac{1}{2}\big(\tfrac{d}{4}\big)^\ell$ for the second. By the same
  argument, $y$ has at least $\frac{1}{2} \left( \frac{d}{4} \right)^{q-1}
  e(y,Y)$ neighbours in $S_I(\tpl{e}{q-1})$ for each $I \subset [q-1]$.
  
  Define, for each $I \subset [q]$, the set $A_I(\tpl{e}{q}) \subset Y$ as
  follows:
  \begin{equation}\begin{split}\label{eq:defA}
    A_I(\tpl{e}{q}): = N(x) \cap S_{I \setminus
    \{q\}}(\tpl{e}{q-1}) \;\, \quad \; &
    \textup{ if } q \in I\\
    A_I(\tpl{e}{q}): = N(y) \cap S_I(\tpl{e}{q-1}) \quad \quad \,\, & \textup{ if } q
    \not\in I.
  \end{split}\end{equation}
  Since $e(x,Y), e(y,Y) \ge d |Y|$, we conclude from~\eqref{eq:boostcalc},
  that we have $|A_I(\tpl{e}{q})| \ge 2 \left( \frac{d}{4} \right)^q |Y|$ for
  every $I \subset [q]$. Moreover,
  %since the set $S_I(\tpl{e}{q-1})$ are disjoint, it follows that
  the sets $A_I(\tpl{e}{q})$  and $A_J(\tpl{e}{q})$ are disjoint
  unless $I \setminus \{q\} = J \setminus \{q\}$. Hence we may choose disjoint
  sets $S_I(\tpl{e}{q}) \subset A_I(\tpl{e}{q})$ with $|S_I(\tpl{e}{q})| =
  \left( \frac{d}{4} \right)^q |Y|$ for each $I \subset [q]$.
  
  Let $S(\tpl{e}{q}) = \big\{ S_I(\tpl{e}{q}) : I \subset [q] \big\}$. We
  claim that this is the desired family; that is, that $S(\tpl{e}{q}) \in
  \cF^{3,\alpha_q |Y|}_{q}(Y)$ and $\tpl{e}{q} \to S(\tpl{e}{q})$. Indeed,
  the sets $S_I(\tpl{e}{q})$ are disjoint, and
  \[|S_I(\tpl{e}{q})|= \left( \frac{d}{4} \right)^q |Y|=\alpha_q |Y|\] for
  each $I \subset [q]$, by construction. Finally, we prove that $\tpl{e}{q}
  \to S(\tpl{e}{q})$, i.e., that $e_i \to_i S(\tpl{e}{q})$ for each $i \in
  [q]$. For $i \le q - 1$, this follows because $\tpl{e}{q-1} \to
  S(\tpl{e}{q-1})$, and
  \[S_I(\tpl{e}{q}) \cup S_{I \cup \{q\}}(\tpl{e}{q}) \subset
  S_I(\tpl{e}{q-1})\] 
  for every $I \subset [q-1]$ by~\eqref{eq:defA}. For $i = q$, it follows
  since $S_I(\tpl{e}{q}) \subset N(x)$ if $q \in I\subset[q]$ and
  $S_I(\tpl{e}{q}) \subset N(y)$ if $q \not\in I\subset[q]$
  by~\eqref{eq:defA}. Hence $ \tpl{e}{q} \to S(\tpl{e}{q})$, as
  required. This completes the induction step, and hence the proof of the
  lemma.
\end{proof}

We can now easily deduce Lemma~\ref{no-p-booster}.

\begin{proof}[Proof of Lemma~\ref{no-p-booster}]
  By  Lemma~\ref{no-boost-tech} (applied with $q = \ell$), it suffices to show
  that if $(X,Y)$ is $(C,\alpha_\ell,\ell)$-Zykov, then
  it is $(C,\alpha)$-rich in copies of $Z_\ell^{3,t}$, where
  $\alpha_\ell=\big(\tfrac{d}{4}\big)^\ell$ and $\alpha = \left(
  \frac{\alpha_\ell}{t} \right)^{2^\ell t}$. In other words, we want to prove
  that if there exists $S \in \cF^{3,\alpha_\ell |Y|}_\ell(Y)$ with
  $\tpl{e}{\ell} \to S$, then
  \[\left| \big\{ S' \in \cF_\ell^{3,t}(Y) \,:\,
  \tpl{e}{\ell} \to S' \big\} \right|  \; \ge \; \alpha |Y|^s\,,\]
  where $s = 2^\ell t$. Indeed, this is true because $|S_I|=\alpha_\ell|Y|$
  for $I\subset[\ell]$, and the number of ways of choosing, for each
  $I\subset[\ell]$, a $t$-subset of $S_I$ is 
  \[\prod_{I\subset[\ell]}\binom{|S_I|}{t} \,=\, \binom{\alpha_\ell |Y|}{t}^{2^\ell}
  \ge \, \left( \frac{\alpha_\ell|Y|}{t} \right)^{2^\ell t} = \, \alpha
  |Y|^s\,,\] as claimed.
\end{proof}

It is now straightforward to prove Proposition~\ref{lem:VC}.

\begin{proof}[Proof of Proposition~\ref{lem:VC}] 
  Let $C,\ell,t \in \NATS$ and $d > 0$, and set $\beta = \left( \frac{d}{4}
  \right)^\ell$ and $\eps = \beta / 2$.  Let $G$ and $(X,Y)$ be as
  described in the statement, so $|N(x) \cap Y| \ge d|Y|$ for every $x \in
  X$. By Lemma~\ref{exists-booster} there exists a
  $(2^\ell,C,\eps,\beta)$-booster tree for $(X,Y)$, and moreover $|\cT|$ is
  bounded as a function of $d$, $\eps$ and $\ell$.

  Recall that the leaves of $\cT$ correspond to a partition of $X$ (and a
  partition of $Y$). If every leaf $(X',Y')$ of $\cT$ is degenerate then
  $\chi(G[X]) \le |\cT|(C+1)=:C'$, where~$C'$ depends only upon $C$, $\ell$
  and $d$.  So we may assume that some leaf $(X',Y') \in V(\cT)$ is not
  degenerate.

  By the definition of a $(2^\ell,C,\eps,\beta)$-booster tree, it follows
  that there is no $(p,C,\eps,\beta)$-booster of $(X',Y')$ for any $p \le
  2^\ell$.  Further, $|N(x)\cap Y'|=d(x,Y')|Y'|\ge d|Y'|$ for every $x\in
  X'$ by Lemma~\ref{exists-booster}. Then, by Lemma~\ref{no-p-booster}
  (applied with $\ell$, $t$ and $d$), $(X',Y')$ is $(C,\alpha')$-rich in
  copies of $Z_\ell^{3,t}$, for some $\alpha' = \alpha'(d,\ell,t) >
  0$. Since (again by Lemma~\ref{exists-booster}) $|Y'| \ge \beta^{|\cT|}
  |Y|$, it follows that $(X,Y)$ is $(C,\alpha)$-rich in copies of
  $Z_\ell^{3,t}$, where $\alpha =\alpha'\beta^{|T|s}$ is a constant
  depending only on $d$, $\ell$ and $t$, as required.
\end{proof}

\section{The proof of Theorem~\ref{mainthm}}\label{proofsec}

In this section we shall complete the proof of Theorem~\ref{mainthm}. As a
warm-up, we begin with the case $r = 3$, which is
an almost immediate consequence of the results of the last four sections.

The following theorem proves Conjecture~\ref{LTconj}. The proof does not use the
Regularity Lemma; it follows from Propositions~\ref{lem:rich-H}
and~\ref{lem:VC}.

\begin{theorem}\label{nearacyclic}
  If $H$ is a near-acyclic graph, then $\delta_\chi(H) = 0$.
\end{theorem}

\begin{proof}
  Let $H$ be a near-acyclic graph (so in particular $\chi(H) = 3$), let $\gamma
  > 0$ be arbitrary, and let $G$ be an $H$-free graph on $n$ vertices, with
  $\delta(G) \ge 2\gamma n$. We shall prove that the chromatic number of $G$ is at
  most $C'$, for some $C' = C'(H,\gamma)$.
  
  First, using Observation~\ref{acy=Zyk}, choose $t \in \NATS$ and a collection
  $T_1,\ldots,T_\ell$ of trees such that $H \subset
  Z_\ell^{3,t}(T_1,\ldots,T_\ell)$. Choose a maximal bipartition $(X,Y)$ of
  $G$, assume without loss of generality that $\chi(G[X]) \ge \chi(G[Y])$, and note
  that $|N(x) \cap Y| \ge \gamma |Y|$ for every $x \in X$.
  
  Let $\alpha > 0$ be given by Proposition~\ref{lem:VC} (applied with
  $\ell$, $t$ and $\gamma$), let $C := 2^{\ell+3} \alpha^{-1} \sum_{i=1}^{\ell}
  |T_i|$, and apply Proposition~\ref{lem:VC}. We obtain a $C' = C'(H,\gamma) >
  0$ such that either $\chi(G) \le 2 \chi\big( G[X] \big) \le 2C'$, or $(X,Y)$
  is $(C,\alpha)$-rich in copies of $Z_\ell^{3,t}$.
  
  In the former case we are done, and so let us assume the latter. By
  Proposition~\ref{lem:rich-H} and our choice of $C$, it follows that
  $Z_\ell^{3,t}(T_1,\ldots,T_\ell) \subset G$. But then $H \subset G$, which is a
  contradiction. Thus $\chi(G)$ is bounded, as claimed.
\end{proof}

The case $r = 3$ of Theorem~\ref{mainthm} now follows from
Proposition~\ref{noforest}, and Theorems~\ref{thm:forest},~\ref{LTborsuk}
and~\ref{nearacyclic}.

\begin{proof}[Proof of the case $r = 3$ of Theorem~\ref{mainthm}] Let $H$ be a
  graph with $\chi(H) = 3$, and recall that $\cM(H)$ denotes the decomposition
  family of $H$. By Proposition~\ref{noforest}, if $\cM(H)$ does not contain a
  forest then $\delta_\chi(H) = \frac{1}{2}$,  and by Theorem~\ref{thm:forest}, if
  $\cM(H)$ does contain a forest then $\delta_\chi(H) \le \frac{1}{3}$.
  
  Now, by Theorem~\ref{LTborsuk}, if $H$ is not near-acyclic then $\delta_\chi(H)
  \ge \frac{1}{3}$, and by Theorem~\ref{nearacyclic}, if $H$ is near-acyclic then
  $\delta_\chi(H) = 0$. Thus $$\delta_\chi(H) \, \in \, \big\{ 0, \, 1/3, \, 1/2
  \big\},$$ where $\delta_\chi(H) \neq \frac{1}{2}$ if and only if $H$ has a
  forest in its decomposition family, and $\delta_\chi(H) = 0$ if and only if $H$
  is near-acyclic, as required.
\end{proof}

The rest of this section is devoted to the proof of the following theorem, which
generalises Theorem~\ref{nearacyclic} to arbitrary $r \ge 3$.

\begin{theorem}\label{r-acyclic}
  Let $H$ be a graph with $\chi(H) = r \ge 3$. If $H$ is $r$-near-acyclic, then
  \[\delta_\chi(H) \,=\, \ds\frac{r-3}{r-2}\,.\]
\end{theorem}

We begin with the lower bound, which follows by essentially the same
construction as in Proposition~\ref{noforest}.

\begin{prop}\label{lower}
  For any graph $H$ with $\chi(H) = r \ge 3$, we have $\delta_\chi(H) \ge
  \frac{r-3}{r-2}$.
\end{prop}

\begin{proof}
  We claim that, for any such $H$, $n_0$ and $C$, there exist $H$-free graphs
  on $n\ge n_0$ vertices, with minimum degree $\frac{r-3}{r-2} n$, and chromatic
  number at least $C$. Recall that we call a graph a $(k,\ell)$-Erd\H{o}s graph
  if it has chromatic number at least $k$ and girth at least $\ell$, and that
  such graphs exist for every $k,\ell \in \NATS$.
  
  Let $G'$ be a $(C,|H|+1)$-Erd\H{o}s graph on at least $n_0$ vertices, and
  let~$G$ be the graph obtained from the complete, balanced $(r-2)$-partite graph on $(r-2)|G'|$
  vertices by replacing one of its partition classes with~$G'$. Then $G$ is
  $H$-free, since every $|H|$-vertex subgraph of $G$ has chromatic number at
  most $r-1$. Moreover, $\delta(G) = \frac{r-3}{r-2} n$
  and $\chi(G) \ge C$, as required.
\end{proof}

We now sketch the proof of the upper bound of Theorem~\ref{r-acyclic}. Let
$G$ be an $n$-vertex, $H$-free graph with minimum degree
$\big(\frac{2r-5}{2r-3} + 3\gamma\big) n$. Let $T_1,\ldots,T_\ell$ be such
that $H\subset Z_\ell^{r,t}(T_1,\ldots,T_\ell)$.
First, we take an $(\eps,d)$-regular partition, using the degree form of
the Regularity Lemma (where $\eps$ and $d$ will be chosen sufficiently
small given $\gamma$). We then construct a second partition $\cP$ of
$V(G)$, similar to that used in the proof of Theorem~\ref{thm:forest}. Our
aim is to show that $\chi(G[X])\le C'$ for each $X\in\cP$.

In the next step, we observe that the minimum degree condition guarantees
that for each $X\in\cP$, there are clusters $Y$ and $Z_1,\ldots,Z_{r-3}$ of the
$(\eps,d)$-regular partition with the following properties. First, for each
$v\in X$ we have $d_Y(v)\ge\gamma |Y|$, and for each
$i\in[r-3]$ we have $d_{Z_i}(v)\ge\big(\tfrac{1}{2}+\gamma\big)|Z_i|$. Second,
$Y,Z_1,\ldots,Z_{r-3}$ forms a clique in the reduced graph of the
$(\eps,d)$-regular partition. 

Now recall that $Z_\ell^{r,t}(T_1,\ldots,T_\ell)$ contains independent sets
$S_I$ for each $I\subset[\ell]$, and independent sets $S_i$ for each
$i\in[r-3]$. The idea now is to show that if $\chi(G[X])\le C'$ does not
hold, then we find a copy of $Z_\ell^{r,t}(T_1,\ldots,T_\ell)$ in which the
trees $T_1,\ldots,T_\ell$ lie in $X$, the independent sets $S_I$ lie in
$Y$, and $S_i$ lies in $Z_i$ for each $i\in[r-3]$, which contradicts the
assumption that $G$ is $H$-free.

In order to achieve this, we work as follows. We apply the paired
VC-dimension argument (Proposition~\ref{lem:VC}) to $(X,Y)$, with constants
$\ell^*$ and $t^*$ which are much larger than $\ell$ and $t$, and a very large
$C^*$. This yields our $C'$ and an $\alpha>0$ such that either $\chi(G[X])\le
C'$ (in which case we are done), or $(X,Y)$ is $(C^*,\alpha)$-rich in copies
of $Z_{\ell^*}^{3,t^*}$.

In the latter case, we apply Lemma~\ref{q-induc} to conclude that $(X,Y)$
is $(C^*,\alpha)$-dense in copies of $Z_{\ell^*}^{3,t^*}$. The main work
of this section then is to show (in Proposition~\ref{3-to-r-rich}) that
this implies that there is an
$S\in\cF_\ell^{r,t}(Y\cup Z_1\cup\cdots\cup Z_{r-3})$ such that $S$ is
$(r,\ell,t,C,\alpha)$-good for $(X,Y\cup Z_1\cup\cdots\cup Z_{r-3})$.
Finally, applying Lemma~\ref{lem:denseembed} we find that there is a copy
of $Z_\ell^{r,t}(T_1,\ldots,T_\ell)$ in $G$.

\smallskip

As just explained, the following proposition is the main missing tool for
the proof of Theorem~\ref{r-acyclic}.

\begin{prop}\label{3-to-r-rich}
  For every $r>3$, $\ell,t \in \NATS$ and $d,\gamma > 0$ there exist
  $\ell^*,t^*\in\NATS$ such that for every $\alpha>0$ and $C\in\NATS$, there
  exist $\eps_1>0$ and $C^*\in\NATS$, such that for every $0<\eps<\eps_1$ the
  following holds.
  
  Let $G$ be a graph, and let $X$, $Y$ and $Z_1,\ldots,Z_{r-3}$ be disjoint
  subsets of $V(G)$, with $|Y| = |Z_j|$ for each $j \in [r-3]$. Let
  $Z:=Z_1\cup\cdots\cup Z_{r-3}$. Suppose that
  $(Y,Z_j)$ and $(Z_i,Z_j)$ are $(\eps,d)$-regular for each $i \ne j$, and that
  $$|N(x) \cap Z_j| \ge \left( \frac{1}{2} + \gamma \right) |Z_j|$$ for every $x
  \in X$ and $j \in [r-3]$.
  
  If $(X,Y)$ is $(C^*,\alpha)$-dense in copies of $Z_{\ell^*}^{3,t^*}$, then
  there is some $S\in\cF_\ell^{r,t}(Y\cup Z)$ such that
  $S$ is $(r,\ell,t,C,\alpha)$-good for $(X,Y\cup Z)$.
\end{prop}

For the proof of this proposition, we combine an application of the
Counting Lemma and two uses of the pigeonhole principle. As a preparation
for these steps we need to show that there exists a family
$S^*\in\cF_{\ell^*}^{3,t^*}$ which is $(3,\ell^*,t^*,C^*,\alpha)$-good for
$(X,Y)$ and `well-behaved' in the following sense. For each of the sets
$S^*_I\subset Y$ given by~$S^*_I$ only a small positive fraction of the
$(r-3)t$-element sets in~$Z$ has a common neighbourhood in $S^*_I$ of less
than~$t$ vertices. To this end we shall use the following lemma.

Recall that for a set
$T$ of vertices in a graph $G$, we write
\[N(T)\colon = \bigcap_{x\in T} N(x)\,.\]

\begin{lemma}\label{countST}
  For all $r,t \in \NATS$ and $\mu,d > 0$, there exist $t^* =
  t^*(r,t,\mu,d) \in \NATS$ and $\eps_0 = \eps_0(r,t,\mu,d) > 0$ such that for
  all $0<\eps<\eps_0$ the following holds.
  
  Let $G$ be a graph, and suppose that $Y$ and $Z_1,\ldots,Z_{r-3}$ are disjoint
  subsets of $V(G)$ such that $(Y,Z_j)$ is $(\eps,d)$-regular for each $j \in
  [r-3]$. Let $Z := Z_1 \cup \ldots \cup Z_{r-3}$, and define
  \[\cB(S) := \Big\{ T \in \binom{Z}{(r-3)t} \,:\, |N(T) \cap S | < t \Big\}\]
  for each $S \subset Y$. Then we have
  \[\cS:=\Big\{ S \in \binom{Y}{t^*} \,:\, |\cB(S)| \ge \mu
  |Z|^{(r-3)t} \Big\}  \; \le \; \sqrt{\eps} |Y|^{t^*}\,.\]
\end{lemma}

\begin{proof}
  Choose~$t^*$ sufficiently large such that
  \begin{equation}\label{eq:countST:t*}
    \Pr\Big( \Bin\big( t^*, (d/2)^{(r-3)t} \big) < t \Big) \le
    \mu \,,
  \end{equation}
   where $\Bin(n,p)$ denotes a random variable with binomial distribution,
   and set
  \begin{equation}\label{eq:countST:eps}
    \eps_0:=\min\Big\{\Big(\frac d2 \Big)^{t^*}, \big(t^*\cdot 2^{t^*}(r-3)\big)^{-2}\Big\} \,.
  \end{equation}

  In the first part of the proof we shall construct a family $\cS'$ of at least
  $\binom{|Y|}{t^*}-\sqrt\eps|Y|^{t^*}$ sets~$S\in\binom{Y}{t^*}$. In the
  second part we will then show that $\cS'\subset
  \binom{Y}{t^*}\setminus\cS$, which proves the lemma. For constructing the
  sets $S\in\cS'$ we proceed inductively and shall choose the vertices
  $v_1,\dots,v_{t^*}$ of~$S$ one by one, in each step $k\in[t^*]$ avoiding
  a set $Y_k\subset Y$ of size at most $\eps 2^k(r-3)|Y|$. Clearly,
  by~\eqref{eq:countST:eps}, this gives at least
  $\binom{|Y|}{t^*}-\sqrt\eps|Y|^{t^*}$ choices for~$S$ as desired.

  Indeed, suppose we have already chosen the vertices
  $v_1,\ldots,v_{k-1}$. In addition we have chosen for each $j\in[r-3]$ a
  partition~$P_j^{k-1}$ of~$Z_j$ with the following property (we shall make
  use of these partitions in part two of the proof): for each
  $I\subset\{v_1,\dots,v_{k-1}\}$ we have chosen a part $P_{j}^{k-1}(I)$ of
  size $(d-\eps)^{|I|}(1-d+\eps)^{k-1-|I|}|Z_j|$ such that
  $P_{j}^{k-1}(I)\subset N(I)$. 
  % is contained in the common neighbourhood of~$I$.
  Now we will explain how~$v_k$ can be chosen together with
  partitions~$P_j^{k}$ satisfying the above conditions. For this purpose
  consider the set $Y_k \subset Y$ of vertices~$y$ such that for some
  $j\in[r-3]$ and some $I\subset\{v_1,\dots,v_{k-1}\}$ we have
  \[ |N(y)\cap P_{j}^{k-1}(I)|< (d-\eps)^k|Z_j| \,,\] where
  $P_j^{0}:=\{Z_j\}$ is the trivial partition of~$Z_j$.  The possible
  choices for~$v_k$ now are the vertices in $Y\setminus Y_k$.  The
  partitions~$P_j^{k}$ with $j\in[r-3]$ are defined as follows. For each
  $I'\subset\{v_1,\dots,v_{k-1}\}$ we choose an arbitrary subset~$P$ of
  $N(v_k)\cap P_j^{k-1}(I)$ with $|P|=(d-\eps)^k|Z_j|$, which is possible
  by the choice of~$v_k$, and set
  \begin{equation*}
    P_j^{k}(I') := P_j^{k-1}(I') \setminus P
    \quad\text{and}\quad
    P_j^{k}\big(I'\cup\{z_k\}\big) := P \,.
  \end{equation*}
  Clearly, the partitions defined in this way satisfy that each part
  $P_j^{k}(I)$ is of size $(d-\eps)^{|I|}(1-d+\eps)^{k-|I|}|Z_j|$
  and that $P_j^k(I)\subset N(I)$ as desired.

  It remains to show that $|Y_k| \le \eps 2^k(r-3)|Y|$ as claimed above.
  If this is not true, then for some $j\in[r-3]$ and $I \subset [k-1]$,
  there exist $\eps |Y|$ vertices in $Y$ which have at most $(d-\eps)
  |P_j^k(I)|$ neighbours in $P_j^{k-1}(I)$. Since $|P_j^{k-1}(I)| \ge
  (d-\eps)^{k-1} |Z_j|\ge\eps |Z_j|$ by~\eqref{eq:countST:eps}, this
  contradicts $(\eps,d)$-regularity of $(Y,Z_j)$.

  We now turn to the second part of the proof: We claim that for every $S \in \cS'$ we
  have $|\cB(S)| < \mu |Z|^{(r-3)t}$.  To see this, simply choose a random
  multiset $T \subset Z$ of size $(r-3)t$, and observe that $N(T) \cap S$
  is given by the intersection of $(r-3)t$ sets $S_1,\ldots,S_{(r-3)t}
  \subset S$ chosen (independently) according to the distribution
  \begin{equation*}
    \Pr\big( S_i = I \big) 
    =\frac{\big|\{z\in Z\,:\,I=N(z)\cap S\}\big|}{|Z|}
   \qquad\text{for $I\subset S$}\,.
  \end{equation*}
  By construction we have $|P^{t^*}_j(I)| = (d -
  \eps)^{|I|} (1-d+\eps)^{t^* - |I|} |Z_j| $ for every $j \in [r-3]$
  and $I \subset S$. Hence 
  \begin{equation*}
    \Pr(I \subset S_i)
      \ge\frac{\big| \bigcup_{j=1}^{r-3} \bigcup_{I\subset I'\subset S}P^{t^*}_j(I')\big|}{|Z|}
      = \sum_{I\subset I'\subset S}(d-\eps)^{|I'|} (1-d+\eps)^{t^* - |I'|} 
      =(d-\eps)^{|I|} \,.
  \end{equation*}
  This implies that for every $I\subset S$, we have $\Pr(I \subset S_i)\ge\Pr(I
  \subset S'_i)$ for the random variable~$S'_i$ with the following distribution: for every $u
  \in S$ we take $u\in S'_i$ independently with
  probability $d - \eps$.
  We conclude that
  \begin{equation*}\begin{split}
    \Pr\big(|S_1 \cap \ldots \cap S_{(r-3)t}| \ge t\big)
    & = \Pr\big( I \subset S_1 \cap \dots \cap S_{(r-3)t} \quad\text{for
      some $I\subset S$ with $|I|\ge t$} \big) \\
    & \ge \Pr\big( I \subset S'_1 \cap \dots \cap S'_{(r-3)t} \quad\text{for
      some $I\subset S$ with $|I|\ge t$} \big) \\
    & = \Pr\big(|S'_1 \cap \ldots \cap S'_{(r-3)t}| \ge t\big)
    = \Pr\Big( \Bin\big( t^*, (d - \eps)^{(r-3)t} \big) \ge t \Big) \\
    &\ge 1-\mu \,,
  \end{split}\end{equation*}
  where the last inequality follows from~\eqref{eq:countST:t*}.
  This proves $|\cB(S)| < \mu |Z|^{(r-3)t}$ and hence finishes the proof of
  the lemma.
\end{proof}

We shall now prove Proposition~\ref{3-to-r-rich}.

\begin{proof}[Proof of Proposition~\ref{3-to-r-rich}]
  We start by defining the constants.
  Given $r>3$, $\ell,t \in \NATS$ and $\gamma,d  > 0$, we set
  \begin{equation}\label{eq:3tor:setalphaell}
    \mu:=\frac{\gamma^{(r-3)t}}{8\big((r-3)t\big)!(r-3)^{(r-3)t}}
    \Big(\frac{d}{2}\Big)^{\binom{r-3}{2}t^2} \quad\text{and}\quad
    \ell^*:=\frac{\ell}{2\mu}\,.
  \end{equation}
  Let $t^*$ and $\eps_0$ be given by Lemma~\ref{countST} with input
  $r, t, \mu':=2^{-\ell^*}\mu, d$. Given $\alpha>0$ and $C$, we choose
  \begin{equation}\label{eq:3tor:setepsCstar}
    \eps_1:=\min\Big(\frac{\alpha^2}{2^{4\ell^*+1}},
    \frac{d\gamma}{4(\gamma+1)(r-3)t},\eps_0\Big)\quad\text{and}\quad
    C^*:=\frac{2^{\ell^*}C}{\alpha\mu}\,.
  \end{equation}
  
  Now let $0<\eps<\eps_1$, let $G$ be a graph, and let $X$, $Y$ and
  $Z_1,\ldots,Z_{r-3}$ be disjoint subsets of $V(G)$ as described in the
  statement, so in particular, $(X,Y)$ is $(C^*,\alpha)$-dense in copies of
  $Z_{\ell^*}^{3,t^*}$. The goal is to
  show that there exists $S\in\cF_\ell^{r,t}(Y\cup Z)$ such that $S$ is
  $(r,\ell,t,C,\alpha)$-good for $(X,Y\cup Z)$.
  
  Our first step is to show that there is a `well-behaved' function
  $S^*\in\cF_{\ell^*}^{3,t^*}(Y)$.
  
  \begin{claim}\label{clm:wellbvd} There is a function
  $S^*\in\cF_{\ell^*}^{3,t^*}(Y)$ which is $(3,\ell^*,t^*,C^*,\alpha)$-good for
  $(X,Y)$ and has the property that for every $I\subset[\ell^*]$, the set
  \[\cB(S^*_I)=\Big\{T\in\binom{Z}{(r-3)t}\colon \big|N(T)\cap S^*_I\big|\le
  t\Big\}\]
  in $\binom{Z}{(r-3)t}$ has size at most $2^{-\ell^*}\mu |Z|^{(r-3)t}$.
  \end{claim}
  
  \begin{claimproof}[Proof of Claim~\ref{clm:wellbvd}]
    By Lemma~\ref{countST} (with input $r,t,\mu'=2^{-\ell^*}\mu,d$), the
    total number of `bad' $t^*$-subsets $S'$ of $Y$, i.e., those for which
    $\cB(S')\ge 2^{-\ell^*}\mu|Z|^{(r-3)t}$, is at most
    $\sqrt{\eps}|Y|^{t^*}$. Let $\cS$ be the set of functions $S^*$ in
    $\cF_{\ell^*}^{3,t^*}(Y)$ which do \emph{not} have the property that
    for every $I\subset[\ell^*]$ we have
    $\cB(S^*_I)<2^{-\ell^*}\mu|Z|^{(r-3)t}$. We can obtain any function
    $S^*$ in $\cS$ by taking a set $I\subset [\ell^*]$ and one of the at
    most $\sqrt{\eps}|Y|^{t^*}$ `bad' $t^*$-sets to be $S^*_I$, and
    choosing the $2^{\ell^*}-1$ remaining sets of $S^*$ in any way from
    $\binom{Y}{t^*}$. It follows that
    \[|\cS|\le
    2^{\ell^*}\sqrt{\eps}|Y|^{t^*}|Y|^{(2^{\ell^*}-1)t^*}=2^{\ell^*}\sqrt{\eps}|Y|^{2^{\ell^*}t^*}\,.\]
    
    Since $(X,Y)$ is $(C^*,\alpha)$-dense in copies of $Z_{\ell^*}^{3,t^*}$,
    there are at least $2^{-\ell^*}\alpha|Y|^{2^{\ell^*}t^*}$ functions in
    $\cF_{\ell^*}^{3,t^*}(Y)$ which are $(3,\ell^*,t^*,C^*,\alpha)$-good for
    $(X,Y)$. Since by~\eqref{eq:3tor:setepsCstar} we have
    $2^{-\ell^*}\alpha> 2^{\ell^*}\sqrt{\eps}$, at
    least one of these functions is not in $\cS$, as required.
  \end{claimproof}
  
  For the remainder of the proof, $S^*$ will be a fixed function satisfying the
  conclusion of Claim~\ref{clm:wellbvd}. Since $S^*$ is
  $(3,\ell^*,t^*,C^*,\alpha)$-good for $(X,Y)$, there exist sets
  \[E^*_1,\ldots,E^*_{\ell^*} \subset E(X), \quad \text{with}  \quad
  \ol{d}(E^*_j) \ge 2^{-\ell^*} \alpha C^* \quad \text{ for each } \quad 1 \le j
  \le \ell^*\,,\] such that for every $e_{1} \in E^*_{1}, \ldots, e_{\ell^*} \in
  E^*_{\ell^*}$, we have $\tpl{e}{\ell^*} \to S^*$.
  
  Our next claim comprises two applications of the pigeonhole
  principle to find a copy of $K_{r-3}(t)$ in $Z$.
  
  \begin{claim}\label{3-to-r:claim2}
    There exists a copy $T$ of $K_{r-3}(t)$ with $t$ vertices in $Z_j$ for each
    $j \in [r-3]$, and a set $L \subset [\ell^*]$ of size $|L| = \ell$ such that:
    \begin{enumerate}[label=\rom]
      \item\label{3-to-r:a} $|N(T) \cap S^*_I | \ge t$  for every $I \subset
      [\ell^*]$,
      \item\label{3-to-r:b} $N(T)$ contains at least $\mu
      |E^*_j|$ edges of $E^*_j$, for each $j \in L$.
    \end{enumerate}
  \end{claim}
  
  \begin{claimproof}[Proof of Claim~\ref{3-to-r:claim2}]
    By assumption, for every $x \in X$ and $j \in [r-3]$ we have
    \[|N(x) \cap Z_j| \ge \left( \frac{1}{2} + \gamma \right)
    |Z_j|\,,\]
    and so each edge $e \in E^*_1 \cup
    \ldots \cup E^*_{\ell^*}$ has at least $\gamma |Z_j|$ common neighbours in
    $Z_j$. By Fact~\ref{prop:subpair}, the common neighbours of $e$ in $Z_i$ and
    $Z_j$ form an $(\eps/\gamma,d-\eps)$-regular pair for each $1\le i<j\le
    r-3$. By~\eqref{eq:3tor:setepsCstar} we have
    $d-\eps-(r-3)t\eps/\gamma>d/2$.
    Hence, applying the Counting Lemma with~$d$
    replaced by $d-\eps$ and~$\eps$ replaced by $\eps/\gamma$ to the graph $H=K_{r-3}(t)$, it follows
    that there are at least 
    \begin{multline*}
        \frac1{\Aut(H)}\Big(d-\eps-\frac\eps\gamma|H|\Big)^{e(H)}\Big(\frac{\gamma|Z|}{r-3}\Big)^{|H|}
        \\ \ge\frac1{\big((r-3)t\big)!}\Big(\frac
        d2\Big)^{\binom{r-3}{2}t^2}\Big(\frac{\gamma|Z|}{r-3}\Big)^{(r-3)t}
        \geByRef{eq:3tor:setalphaell} 8\mu |Z|^{(r-3)t}
    \end{multline*}
    copies of $K_{r-3}(t)$ in
    $N(e) \cap Z$, each with $t$ vertices in each $Z_j$.
   
    There are therefore, for each $j \in [\ell^*]$, at least $8\mu
    |Z|^{(r-3)t} |E^*_j|$ pairs $(e,T)$, where $e \in E^*_j$ and $T$ is a copy
    of $K_{r-3}(t)$ as described, such that $T \subset N(e)$, or equivalently $e
    \subset N(T)$. Since we have
    \[8\mu
    |Z|^{(r-3)t} |E^*_j|=4 \mu |Z|^{(r-3)t}|E^*_j|+4 \mu
    |E^*_j||Z|^{(r-3)t}\,,\]
    by the pigeonhole principle, it follows that
    there are at least $4 \mu |Z|^{(r-3)t}$ copies of $K_{r-3}(t)$ in~$Z$ each of which
    has at least $4 \mu |E^*_j|$ edges of $E^*_j$ in its common
    neighbourhood. Let us denote by $\cT_j$ the collection of such copies of
    $K_{r-3}(t)$. For a copy $T$ of $K_{r-3}(t)$, let $L(T) = \big\{ j : T
    \in \cT_j \big\}$.
    
    We claim that there is a set $\cT$ containing at least $2\mu
    |Z|^{(r-3)t}$ copies $T$ of $K_{r-3}(t)$ in $Z$, each with $|L(T)|\ge\ell$.
    Indeed, this follows once again by the pigeonhole principle, since there are at least
    \[\ell^*\cdot 4 \mu |Z|^{(r-3)t} \,\eqByRef{eq:3tor:setalphaell}\, \ell
    |Z|^{(r-3)t}+\ell^*\cdot 2\mu|Z|^{(r-3)t}\]
    pairs $(T,j)$ with $T \in \cT_j$.
    
    Now, recall that $S^*$ satisfies the conclusion of Claim~\ref{clm:wellbvd},
    i.e., for each $I\subset[\ell^*]$, there are at most
    $2^{-\ell^*}\mu|Z|^{(r-3)t}$ sets $T \in \binom{Z}{(r-3)t}$ such
    that $|N(T) \cap S^*_I | \le t$. Since $|\cT|\ge 2\mu|Z|^{(r-3)t}$, there
    is a copy $T$ of $K_{r-3}(t)\in\cT$ such that for each $I\subset[\ell^*]$,
    we have $|N(T)\cap S^*_I|\ge t$. If we let~$L$ be any subset of $L(T)$ of size
    $\ell$, then $T$ and $L$ satisfy the conclusions of the claim.
  \end{claimproof}
  
  Let $T$ and $L$ be as given by Claim~\ref{3-to-r:claim2} and for each $j \in L$ let
  $E_j\subset X$ be a set of $\mu|E^*_j|$ edges of $E^*_j$ contained in $N(T)$ as
  promised by~Claim~\ref{3-to-r:claim2}\ref{3-to-r:b}.
  We construct a
  function $S \in \cF_\ell^{r,t}(Y)$ by choosing, for each $I \subset L$, a
  subset $S_I \subset S^*_I$ of size~$t$ in $N(T)\cap Y$ (which is possible
  by Claim~\ref{3-to-r:claim2}\ref{3-to-r:a}), and letting the sets $S_i$,
  $i\in[r-3]$, be the parts of $T$.
  
 \begin{claim}\label{3-to-r:claim3}
    $S$ is $(r,\ell,t,C,\alpha)$-good for $(X,Y \cup Z)$.
  \end{claim}
  \begin{claimproof}[Proof of Claim~\ref{3-to-r:claim3}]
    Recall that $|L| = \ell$, and assume without loss of generality that $L
    = \{1,\ldots,\ell\}$. 
    By the choice of~$T$ and the definition of the sets $S_I$ with $I
    \subset L$ and the sets~$S_i$ with
    $i\in[r-3]$, we have that $S_i$ is completely adjacent to each $S_{i'}$
    with $i\neq i'$, to each~$S_I$, and to each edge $e\in \bigcup_{j\in L} E_j$.
    Since $\tpl{e}{\ell^*} \to S^*$ for each
    $\tpl{e}{\ell^*} \in E^*_1 \times \ldots \times E^*_{\ell^*}$, it
    follows that $\tpl{e}{\ell} \to S$ for each $\tpl{e}{\ell} \in E_1
    \times \ldots \times E_{\ell}$. Finally, for each $j\in L$, since
    $|E_j|\ge\mu|E^*_j|$, we have
    \[\ol{d}(E_j)\ge \mu\ol{d}(E^*_j)\ge \mu2^{-\ell^*}\alpha
    C^*\eqByRef{eq:3tor:setepsCstar} C\,,\]
    as required.
  \end{claimproof}
  
  Thus there exists a function $S \in \cF_\ell^{r,t}(Y)$ which is
  $(r,\ell,t,C,\alpha)$-good for $(X,Y \cup Z)$, as required.
\end{proof}

\begin{remark}
  It is possible to strengthen the conclusion of
  Proposition~\ref{3-to-r-rich}: under the same conditions, $(X,Y\cup
  Z_1\cup\cdots\cup Z_{r-3})$ is $(C,\alpha')$-dense in copies of
  $Z_{\ell}^{r,t}$, for some
  $\alpha'=\alpha'(r,\ell,t,d,\gamma,\alpha)>0$. To see this, observe
  that the proofs of Claims~\ref{clm:wellbvd} and~\ref{3-to-r:claim2} both
  in fact yield a positive density of functions $S^*$ in
  $\cF_{\ell^*}^{3,t^*}(Y)$ and of copies $T$ of $K_{r-3}(t)$, respectively.
  From any such $S^*$ and $T$ can be obtained a function $S$
  which is $(r,\ell,t,C,\alpha)$-good for $(X,Y \cup Z)$.
\end{remark}

We can now deduce Theorem~\ref{r-acyclic}.

\begin{proof}[Proof of Theorem~\ref{r-acyclic}]
  The lower bound is given by Proposition~\ref{lower}, so we are only
  required to prove the upper bound. Let $H$ be an $r$-near-acyclic graph,
  with $r \ge 4$, and let $\gamma > 0$. Because $H$ is $r$-near-acyclic, by
  Observation~\ref{acy=Zyk} there exist trees $T_1,\ldots,T_\ell$ and a
  number $t\in\NATS$ such that $H\subset
  Z_\ell^{r,t}(T_1,\ldots,T_\ell)$. We now set constants as follows. First,
  we choose $d=\gamma$. Given $r$, $\ell$, $t$, $d$ and $\gamma$,
  Proposition~\ref{3-to-r-rich} returns integers $\ell^*$ and $t^*$.  Now
  Proposition~\ref{lem:VC}, with input $\ell^*,t^*$ and $d$, returns
  $\alpha>0$. Next, consistent with Lemma~\ref{lem:denseembed} we set
  $C:=2^{\ell+3}\alpha^{-1}\sum_{i=1}^\ell |T_i|$. Feeding $\alpha$ and $C$
  into Proposition~\ref{3-to-r-rich} yields $\eps_1>0$ and $C^*$. Putting
  $C^*$ into Proposition~\ref{lem:VC} yields a constant $C'$. We choose
  \begin{equation}\label{eq:r-acyclic:eps}
    k_0:=2r/\gamma\quad\text{and}\quad \eps:=\min(\eps_1,\gamma)\,.
  \end{equation}
  Finally, from the minimum degree form of the Szemer\'edi Regularity
  Lemma, with input $\eps$, $d$, $\delta=(\tfrac{r-3}{r-2}+3\gamma)$ and
  $k_0$, we obtain a constant~$k_1$.
  
  Let $G$ be an $H$-free graph on $n>k_1$ vertices, with
  $\delta(G) \ge \left( \frac{r-3}{r-2} + 3\gamma \right) n$. We shall prove
  that $\chi(G)\le 2\cdot2^{2k_1}C'$.
  First, applying the minimum degree form of the Szemer\'edi Regularity Lemma,
  we obtain a partition $V_0 \cup \ldots \cup V_k$ of $V(G)$, with reduced graph
  $R$, where $\delta(R) \ge \left( \frac{r-3}{r-2} + \gamma \right) k$. We form
  a second partition by setting
  \begin{eqnarray*}
    X(I_1,I_2) \,:=\, \bigg\{ v \in V(G) \,\colon\,   i \in I_1 &
    \Leftrightarrow & |N(v) \cap V_i| \ge \gamma |V_i|\\
    \textup{ and} \quad  i \in I_2 & \Leftrightarrow &  |N(v) \cap V_i| \ge
    \left( \frac{1}{2} + \gamma \right) |V_i| \bigg\}
  \end{eqnarray*}
  for each pair of sets $I_2 \subset I_1 \subset [k]$. It obviously suffices to
  establish that for each $I_1$ and $I_2$ we have
  $\chi\big(G[X(I_1,I_2)]\big)\le 2C'$.
  
  Hence let $I_2 \subset I_1 \subset [k]$ be fixed. Since $\chi\big(G[X(I_1,I_2)]\big)\le 2C'$
  is obvious when $X(I_1,I_2)$ is empty, assume it is non-empty. Then the
  minimum degree condition on $G$ allows us to establish the following claim.
  
 \begin{claim}\label{r-acyclic:claim}
    There exist distinct clusters $Y,Y' \in I_1$ and
    $Z_1,Z'_1,\ldots,Z_{r-3},Z'_{r-3} \in I_2$ such that
    $(Y,Z_i),(Y',Z'_i),(Z_i,Z_j)$ and $(Z'_i,Z'_j)$ are $(\eps,d)$-regular
    for
    every pair $\{i,j\} \subset [r-3]$.
  \end{claim}
  \begin{claimproof}[Proof of Claim~\ref{r-acyclic:claim}]
    Let $x$ be any vertex in $X(I_1,I_2)$, and let $m=|V_1|=\cdots=|V_k|$. By
    the definition of $X(I_1,I_2)$, we have $|N(x) \cap V_i| \ge
    \gamma m$ iff $i\in I_1$, and thus
    \[\big(\tfrac{r-3}{r-2}+3\gamma\big)n\le\delta(G)\le d(x)\le \eps
    n+\big(k-|I_1|\big)\gamma m+|I_1|m\le (\eps+\gamma)n+|I_1|\tfrac{n}{k}\,.\]
    Since by~\eqref{eq:r-acyclic:eps} we have $\eps<\gamma$, we deduce
    $|I_1|\ge \big(\frac{r-3}{r-2}+\gamma\big)k$. Similarly, 
    we have $|N(x) \cap V_i| \ge
    (\frac12+\gamma)m$ iff $i\in I_2$ and therefore
    \[\big(\tfrac{r-3}{r-2}+3\gamma\big)n\le d(x)\le \eps
    n+\big(k-|I_2|\big)\big(\tfrac{1}{2}+\gamma\big)m+|I_2|m\le
    (\eps+ \tfrac12+\gamma) n+|I_2|\tfrac{n}{2k}\,,\]
    from which we obtain $|I_2|\ge\big(\frac{r-4}{r-2}+\gamma\big)k$.
    
    Since $\delta(R) \ge  \big( \frac{r-3}{r-2} + \gamma \big) k$, each cluster
    in~$R$ has at most $\big(\frac{1}{r-2}-\gamma\big)k$ non-neighbours.
    It follows that
    \[\delta\big( R[I_2] \big)  \ge  |I_2| - \tfrac{k}{r-2} + \gamma k  \ge
    \big( \tfrac{r-5}{r-4} + \gamma \big) |I_2|\,,\]
    so by Tur\'an's theorem, $R[I_2]$ contains a copy of $K_{r-3}$. We let
    its clusters be $Z_1,\ldots,Z_{r-3}$. Since each $Z_i$ is
    non-adjacent to at most $\big(\frac{1}{r-2}-\gamma\big)k$ cluster in
    $I_1$, there is a cluster $Y$ in $I_1$ adjacent in $R$ to each $Z_i$ with
    $i\in[r-3]$.
    Since $k\ge k_0$, by~\eqref{eq:r-acyclic:eps} we have $\gamma
    k-(r-2)\ge\gamma k/2$ and therefore
    \[\delta\big( R[I_2\setminus\{Y,Z_1,\ldots,Z_{r-3}\}] \big)  \ge  |I_2| - \tfrac{k}{r-2} + \tfrac12\gamma k  \ge
    \big( \tfrac{r-5}{r-4} + \tfrac12\gamma \big) \big|I_2\setminus\{Y,Z_1,\ldots,Z_{r-3}\}\big|\,.\]
    Thus we can again apply Tur\'an's theorem to
    $R[I_2\setminus\{Y,Z_1,\ldots,Z_{r-3}\}]$ to obtain a clique
    $Z'_1,\ldots,Z'_{r-3}$ in $I_2$, which has a common
    neighbour $Y'\in I_1\setminus\{Y,Z_1,\ldots,Z_{r-3}\}$, as required.
  \end{claimproof}
  
  Let $Y,Y' \in I_1$ and $Z_1,Z'_1,\ldots,Z_{r-3},Z'_{r-3} \subset
  I_2$ be the clusters given by Claim~\ref{r-acyclic:claim}. Let
  $X=X(I_1,I_2)\cap (Y'\cup
  Z'_1\cup\cdots\cup Z'_{r-3})$, and $X'=X(I_1,I_2)\setminus X$. Observe that
  $X,Y,Z_1,\ldots,Z_{r-3}$ are pairwise disjoint (as are
  $X',Y',Z'_1,\ldots,Z'_{r-3}$). Our goal now is to show that $\chi\big(G[X]\big)\le C'$.
  Since an analogous argument gives $\chi\big(G[X']\big)\le C'$ and we have
  $X(I_1,I_2)=X\dcup X'$, this will imply $\chi\big(G[X(I_1,I_2)]\big)\le
  2C'$, and thus complete the proof.

  We apply Proposition~\ref{lem:VC}, with input $\ell^*$, $t^*$, $d$ and
  $C^*$, to $(X,Y)$. Observe that, since $Y\in I_1$ and $X\subset
  X(I_1,I_2)$, we have $|N(x)\cap
  Y|\ge d|Y|$ for each $x\in X$. Recall that $\alpha$ and $C'$ were defined such
  that the conclusion of Proposition~\ref{lem:VC} is the following. Either
  $\chi\big(G[X]\big)\le C'$, or $(X,Y)$ is $(C^*,\alpha)$-rich in copies of
  $Z_{\ell^*}^{3,t^*}$. In the first case we are done, so we assume the latter.
  We  will show that this contradicts our assumption that $G$
  is $H$-free.
  
  By Lemma~\ref{q-induc} the pair $(X,Y)$ is
  $(C^*,\alpha)$-dense in copies of $Z_{\ell^*}^{3,t^*}$. We now apply
  Proposition~\ref{3-to-r-rich}, with input $r$, $\ell$, $t$, $d$,
  $\gamma$, $\alpha$, $C$, and~$\eps$
  to $X,Y,Z_1,\ldots,Z_{r-3}$. Observe that since $Z_1,\ldots,Z_{r-3}\in I_2$,
  we have $|N(x)\cap Z_i|\ge (\frac{1}{2}+\gamma)|Z_i|$ for each $x\in X$ and
  $i\in[r-3]$. Moreover, by Claim~\ref{r-acyclic:claim}, any pair of $Y,Z_1,\ldots,Z_{r-3}$ is
  $(\eps,d)$-regular. Recall that~$\ell^*$, $t^*$, $\eps_1$ and~$C^*$ were defined
  such that the conclusion of Proposition~\ref{3-to-r-rich} is that there exists
  a function $S\in\cF_\ell^{r,t}(Y\cup Z_1\cup\cdots\cup Z_{r-3})$ which is
  $(r,\ell,t,C,\alpha)$-good for $(X,Y\cup Z_1\cup\cdots\cup Z_{r-3})$.
  Finally, we apply Lemma~\ref{lem:denseembed}, with input
  $r,\ell,t,\alpha$ and $T_1,\ldots,T_\ell$, to $X$ and $Y\cup Z_1\cup\cdots\cup
  Z_{r-3}$. By the definition of~$C$, this lemma gives that
  $H\subset Z_\ell^{r,t}(T_1,\ldots,T_\ell)$ is contained in $G$, a contradiction.
\end{proof}

Finally, we put the pieces together and complete the proof of Theorem~\ref{mainthm}.

\begin{proof}[Proof of Theorem~\ref{mainthm}]
  Let $H$ be a graph with $\chi(H) = r \ge 3$, and recall that $\cM(H)$ denotes
  the decomposition family of $H$. By Proposition~\ref{noforest}, if $\cM(H)$
  does not contain a forest then $\delta_\chi(H) = \frac{r-2}{r-1}$, and by
  Theorem~\ref{thm:forest}, if $\cM(H)$ does contain a forest then
  $\delta_\chi(H) \le \frac{2r-5}{2r-3}$.
  
  Now, by Theorem~\ref{thm:borsuk}, if $H$ is not $r$-near-acyclic then
  $\delta_\chi(H) \ge \frac{2r-5}{2r-3}$, and by Theorem~\ref{r-acyclic}, if $H$
  is $r$-near-acyclic then $\delta_\chi(H) = \frac{r-3}{r-2}$. Thus
  $$\delta_\chi(H) \, \in \, \left\{ \frac{r-3}{r-2}, \, \frac{2r-5}{2r-3}, \,
  \frac{r-2}{r-1} \right\},$$ where $\delta_\chi(H) \neq \frac{r-2}{r-1}$ if and
  only if $H$ has a forest in its decomposition family, and $\delta_\chi(H) =
  \frac{r-3}{r-2}$ if and only if $H$ is $r$-near-acyclic, as required.
\end{proof}

\section{Open questions}\label{probsec}

Although we have determined $\delta_\chi(H)$ for every graph $H$, there are
still many important questions left unresolved. In this section we shall discuss
some of these. We begin by conjecturing that the assumption on the minimum
degree can be weakened to force the boundedness of the chromatic number, as
Brandt and Thomass\'e~\cite{BraTho} proved in the case of the triangle.

\begin{conj}
For every graph $H$ with $\delta_\chi(H) = \lambda(H)$, there exists a constant
$C(H)$ such that the following holds. If $G$ is an $H$-free graph on $n$
vertices and $\delta(G) > \lambda(H)n$, then $\chi(G) \le C(H)$.
\end{conj}

We mention that an analogous statement is not true for~$H$ with
$\delta_\chi(H)\in\{\theta(H),\pi(H)\}$ as a simple modification of our
constructions for Propositions~\ref{noforest} and~\ref{lower} shows: we merely
need to make the partite graphs used in these constructions slightly unbalanced
and to guarantee that the Erd\H{o}s graphs cover the whole partition class they
are pasted into and have a sufficient minimum degree.

For graphs $H$ with $\delta_\chi(H) = 0$ one could still ask
whether the minimum degree condition can be weakened to some function $f(n) =
o(n)$. The following well-known fact shows that this is not the case.

\begin{prop}\label{o(n)}
  Let $H$ be a graph with $\chi(H) \ge 3$, and let $f(n) = o(n)$. For every $C$
  and~$n_1$, there exist $H$-free graphs $G$ on at least $n_1$ vertices with
  $\delta(G) \ge f\big(v(G)\big)$ and $\chi(G)\ge C$.
\end{prop}
\begin{proof}
  Given $H$, $f$, $C$ and $n_1$, let $G_0$ be a
  $(C,v(H)+1)$-Erd\H{o}s graph. Without loss of generality, we may assume
  $\delta(G_0)\ge 1$. Let $n_0$ be such that $f(n)\le n/v(G_0)$ for each $n\ge
  n_0$. Let $G$ be obtained from $G_0$ by blowing up each vertex to a set of
  size $\max(n_0,n_1)$. Then $G$ has at least $n_1$ vertices, and we
  have $\delta(G)\ge v(G)/v(G_0)\ge f\big(v(G)\big)$. Since $G_0$ contains
  no cycle on $v(H)$ or fewer vertices, $G$ contains no odd cycle with $v(H)$
  or fewer vertices. In particular, every $v(H)$-vertex subgraph of $G$ is
  bipartite, and hence $G$ is $H$-free.
\end{proof}
 
Proposition~\ref{o(n)} also implies that for graphs~$H$ with~$\delta_\chi(H)=0$
the upper bound on $\chi(G)$ for $H$-free graphs~$G$ with $\delta(G)\ge\eps n$
increases as $\eps$ goes to zero.
This suggests the following problem. Set 
\begin{align*}
& \delta_\chi(H,k) \; := \; \inf \Big\{ d \,:\, \delta(G) \ge d |G|
\;\textup{ and }\; H \not\subset G \;\; \Rightarrow \;\; \chi(G) \le k\Big\}\,,
\end{align*} 
or, equivalently,
\begin{align*}
& \chi_\delta(H,d) \; := \; \max \Big\{ \chi(G)  \,:\, \delta(G) \ge d
|G| \;\textup{ and }\; H \not\subset G \Big\}\,,
\end{align*} 
and call this the \emph{chromatic profile} of $H$.

\begin{prob}
Determine the chromatic profile for every graph $H$.
\end{prob}

As noted in the Introduction, we have, by the results of Andr\'asfai, Erd\H{o}s
and S\'os~\cite{AES}, Brandt and Thomass\'e~\cite{BraTho}, H\"aggkvist~\cite{Hagg}
and Jin~\cite{Jin}, that \[\delta_\chi(K_3,2) = \frac{2}{5}, \quad
\delta_\chi(K_3,3) = \frac{10}{29} \quad \textup{ and } \quad \delta_\chi(K_3,k)
= \frac{1}{3} \quad \textup{for every $k \ge 4$.}\] We remark that this problem
was also asked by Erd\H{o}s and Simonovits~\cite{ES}, who remarked that it
seemed (in full generality) `too complicated' to study; despite the progress
made in recent years, we still expect it to be extremely difficult. Note that
although our results give explicit upper bounds on $\chi_\delta(H,d)$ for
every graph $H$, even in the case $\delta_\chi(H)=0$, where we do not use the
Szemer\'edi Regularity Lemma, these bounds are very weak.

\smallskip

{\L}uczak and Thomass\'e~\cite{LT} suggested the following more general problem.
Given a (without loss of generality monotone) family $\cF$ of graphs, we define
\begin{align*}
& \delta_\chi(\cF) \; := \; \inf \Big\{ \delta \,:\, \exists\, C = C(\cF,\delta)
\textup{ such that if } G\in\cF \textup{ is a graph on $n$ vertices } \\ & 
\hspace{5cm} \textup{ with } \delta(G) \ge\delta n  \textup{, then } \chi(G) \le
C\Big\}\,.
\end{align*}
\begin{prob}
  What values can $\delta_\chi(\cF)$ take?
\end{prob}

Our results settle this question
completely when $\cF$ is defined by finitely many minimal forbidden
subgraphs (in which case $\delta_\chi(\cF)$ is precisely the minimum of
$\delta_\chi(H)$ over all minimal forbidden subgraphs $H$). For families $\cF$
defined by infinitely many forbidden subgraphs, however, this minimum provides only an
upper bound on $\delta_\chi(\cF)$.

{\L}uczak and Thomass\'e~\cite{LT} suggested in particular to determine
$\delta_\chi(\cB)$, where $\cB$ is the family of graphs $G$ such that for every
vertex $v\in G$, the graph $G\big[N(v)\big]$ is bipartite (as a natural
generalisation of the family of triangle-free graphs, in which every
neighbourhood is an independent set). This family is indeed defined by
infinitely many forbidden subgraphs: to be precise, by the odd wheels. {\L}uczak and Thomass\'e
gave a construction showing that $\delta_\chi(\cB)\ge\frac{1}{2}$, and
conjectured that $\delta_\chi(\cB)=\frac{1}{2}$. Since the wheel $W_5$ (i.e.,\
the graph obtained from $C_5$ by adding a vertex adjacent to all its vertices) is a
forbidden graph for $\cB$, and $\delta_\chi(W_5)=\frac{1}{2}$ by
Theorem~\ref{mainthm}, our results confirm that their conjecture is true.

\smallskip

One can generalise the concept of chromatic
threshold to uniform hypergraphs. Recently, Balogh, Butterfield, Hu, Lenz and
Mubayi~\cite{BBHLM} extended the {\L}uczak-Thomass\'e method
to uniform hypergraphs $\cH$, and thereby proved that $\delta_\chi(\cH) = 0$ for
a large family of such $\cH$. To quote from their paper, `Many open problems
remain'.

\smallskip

Finally, we would like to introduce a new class of problems relating to the
chromatic threshold. There has been a recent trend in Combinatorics towards
proving `random analogues' of extremal results in Graph Theory and Additive Number
Theory (see, for example, the recent breakthroughs of Conlon and
Gowers~\cite{ConGow} and Schacht~\cite{Schacht:KLR}). We propose the following
variation on this theme: for each graph $H$ and every function $p = p(n) \in
[0,1]$, define
\begin{align*}
& \delta_\chi\big( H, p \big) \; := \; \inf \Big\{ d \,:\, \, \exists\,
C(H,d) \text{ such that for } G=G_{n,p}, \text{ asymptotically almost surely,}\\ & \hspace{4cm}
\textup{if } G' \subset G, \; \delta(G') \ge d pn \, \textup{ and } \,
H \not\subset G', \textup{ then } \chi(G') \le C(H,d) \Big\},
\end{align*} 
where $G_{n,p}$ is the Erd\H{o}s-R\'enyi random graph. Note that when $p(n) =
1$, we recover the definition of $\delta_\chi(H)$.

\begin{prob}
  Determine $\delta_\chi(H,p)$ for every graph $H$, and every $p = p(n)$.
\end{prob}

In a forthcoming paper~\cite{CTRG} we intend to show that for every
constant $p>0$ and every graph $H$, we have
$\delta_\chi(H)=\delta_\chi(H,p)$. This is of course trivial in the case
$\delta_\chi(H)=0$, when it follows from the results of this paper together
with the well-known fact that for constant $p$, the minimum degree of
$G_{n,p}$ is asymptotically almost surely at least $pn/2$. In the case
$\delta_\chi(H)>0$, the result is not trivial: but much of the
machinery
developed in this paper can be used unchanged. The following
construction shows
that the result is best possible, in the sense that it
fails to hold for $p= o(1)$.

\begin{theorem}\label{thm:randomconst} 
Let $r \ge 3$ and $C \in \NATS$, and let $H$ be a graph with $\chi(H) = r$ and $\delta_\chi(H) \ge \lambda(H) = \tfrac{2r-5}{2r-3}$. If $\eps > 0$ is sufficiently small, the following holds. If $n^{-\eps} < p < \eps^2$, then asymptotically almost
  surely the graph $G=G_{n,p}$ contains an $H$-free subgraph $G'$
  with $\chi(G') \ge C$ and $\delta(G') \ge (1-\eps)\tfrac{r-2}{r-1}pn$.
\end{theorem}

\begin{proof}[Proof (sketch)]
  Given $r$, $H$ and $C$, we let $F$ be a fixed $(C,v(H)+1)$-Erd\H{o}s graph. We
  choose a sufficiently small $\eps>0$.

  We now construct an $H$-free subgraph~$G'$ of $G=G_{n,p}$ as follows. Let
  $V_1,\ldots,V_{r-1}$ be an arbitrary balanced partition of $[n]$. We fix
  a copy of $F$ within $G[V_1]$ (which exists asymptotically almost
  surely).  Then we delete all edges within each part $V_i$ with
  $i\in[r-1]$, except those in the copy of $F$. Moreover, for each pair of
  vertices $u,v\in V(F)$, we delete the edges from $u$ and $v$ to the common
  neighbours of $u$ and $v$ in each of $V_2,\ldots,V_{r-1}$.

  It follows that $\chi(G')\ge C$ and that $G'$ asymptotically almost
  surely has minimum degree $(1-\eps)\tfrac{r-2}{r-1}pn$. In addition it
  can easily be checked from our characterisation of graphs $H$ with
  $\delta_\chi(H)\ge\lambda(H)$ that~$G'$ is also $H$-free.
\end{proof}

Theorem~\ref{thm:randomconst} can be significantly strengthened, and we intend
to do so in~\cite{CTRG}. However, results of Kohayakawa, R\"odl and
Schacht~\cite{KohRodlSchacht} show that
we cannot increase the value $\frac{r-2}{r-1}$ in the minimum degree, i.e, $\delta_\chi(H,p) \le \pi(H)$. Thus, by Theorem~\ref{thm:randomconst}, and in
contrast to the $p=\Theta(1)$ case, if $\delta_\chi(H) \ge \lambda(H)$ and
$n^{-o(1)}<p= o(1)$, then $\delta_\chi(H,p)=\frac{r-2}{r-1}=\pi(H)$.

%%%% BIBLIOGRAPHY %%%%%%%%%%%%%%%%%%%%%%%%%%%%%%%%%%%%%%%%%%%%%%%%%%%%%

\bibliographystyle{amsplain_yk} 
\bibliography{ChromaticThreshold}

\end{document}